\documentclass[11pt]{amsart}
\usepackage[margin=1.5in]{geometry}
\usepackage{cite}
\usepackage[utf8]{inputenc}
\usepackage[T1]{fontenc}
\usepackage[english]{babel}
\usepackage{amsmath,amsfonts,amsthm, mathrsfs}
\usepackage{hyperref}
\usepackage{tikz-cd}
\usepackage[colorinlistoftodos]{todonotes}
\usepackage{comment}
\usepackage{subcaption}
\usepackage{amssymb}
\theoremstyle{plain}
\newtheorem{theorem}{Theorem}[section]
\newtheorem{corollary}[theorem]{Corollary}
\newtheorem{lemma}[theorem]{Lemma}
\newtheorem{proposition}[theorem]{Proposition}
\newtheorem{conjecture}[theorem]{Conjecture}
\newtheorem*{corollary*}{Corollary}

\theoremstyle{definition}
\newtheorem{definition}[theorem]{Definition}
\newtheorem{question}[theorem]{Question} 

\newtheorem{remark}[theorem]{Remark}
\numberwithin{equation}{section}
\numberwithin{figure}{section}
\numberwithin{table}{section}

\DeclareMathOperator{\E}{\mathbb E}
\renewcommand{\P}{\operatorname{\mathbb P}}

\DeclareMathOperator{\Tr}{Tr}
\DeclareMathOperator{\Pois}{Pois}

\def\R{\mathbb R}
\def\Q{\mathbb Q}
\def\Z{\mathbb Z}
\def\C{\mathbb C}
\def\N{\mathbb N}
\def\F{\mathbb F}

\title[Universality for random polynomials over finite fields]{Universality for low degree factors of random polynomials over finite fields}

\author{Jimmy He}
\address{Department of Mathematics, MIT, Cambridge, MA  02139}
\email{jimmyhe@mit.edu}

\author{Huy Tuan Pham}
\address{Department of Mathematics, Stanford University, Stanford, CA  94305}
\email{huypham@stanford.edu}

\author{Max Wenqiang Xu}
\address{Department of Mathematics, Stanford University, Stanford, CA  94305}
\email{maxxu@stanford.edu}


\begin{document}
\maketitle
\begin{abstract}
We show that the counts of low degree irreducible factors of a random polynomial $f$ over $\mathbb{F}_q$ with independent but non-uniform coefficients behave like that of a uniform random polynomial, exhibiting a form of universality for random polynomials over finite fields. Our strongest results require various assumptions on the parameters, but we are able to obtain results requiring only $q=p$ a prime with $p\leq \exp({n^{1/13}})$ where $n$ is the degree of the polynomial. Our proofs use Fourier analysis, and rely on tools recently applied by Breuillard and Varj\'u \cite{BV19poly,BV19walk} to study the $ax+b$ process, which show equidistribution for $f(\alpha)$ at a single point. We extend this to handle multiple roots and the Hasse derivatives of $f$, which allow us to study the irreducible factors with multiplicity.
\end{abstract}

\section{Introduction}
Let $\overline{f}(x)=x^{n}+\sum_{i=0}^{n-1} \varepsilon_i x^i$ be a random monic polynomial with independent uniformly distributed coefficients in a finite field $\F_q$, and consider $N_i'(\overline{f})$, the number of irreducible factors of $\overline{f}$ in $\F_q[x]$ of degree $i$. The moments can be explicitly computed using generating function arguments, and the work of Arratia, Barbour and Tavar\'e established strong asymptotics for these random variables \cite{ABT93}.

In this paper, we study random polynomials $f(x)=\sum_{i=0}^n \varepsilon_i x^i$, where the $\varepsilon_i$ are independent but no longer uniformly distributed in $\F_q$. We establish results that suggest the distribution of the $N_i'(\overline{f})$ are universal, at least when $i$ is not too large compared to $n$. For example, we show the joint distribution of the counts under the non-uniform and uniform models, $N_i'(f)$ and $N_i'(\overline{f})$, for $i\leq n^{\frac{1}{13}}$, are close in total variation distance as long as $q=p$ is prime and $p\leq e^{n^{\frac{1}{13}}}$. We also obtain results for $q=p^e$ as long as $q$ is not too large compared to $n$.

We study these polynomials $f(x)$ using Fourier-analytic methods. Note that the value of $f(\alpha)$ for some $\alpha\in \F_{q^d}$ can be viewed as the state of a random walk defined by $X_{t+1}=\alpha X_t+\varepsilon_{t+1}$ where $\varepsilon_{t+1}$ is drawn from some non trivial distribution. This is known as the $ax+b$ process or the Chung--Diaconis--Graham process, and we use recent tools developed to study this process in \cite{BV19poly,BV19walk}. For this reason, our strongest results apply only when the coefficients lie in $\F_p$. 

This also leads us to study the distribution of $N_i(f)$, the number of \emph{distinct} irreducible factors of degree $i$, since these values are more closely related to the equidistribution of $X_{n+1}=f(\alpha)$. To study the $N_i'(f)$, which count irreducible factors with multiplicity, we instead study the equidistribution of the values of $f$ along with its Hasse derivatives. This requires extending the tools developed in \cite{BV19poly,BV19walk}.

Finally, we remark that certain universality results on random matrices over finite fields, and especially on their eigenvalues (which are of course the roots of the characteristic polynomial), seem to be at least morally related to our results. Recent works \cite{LMN21, eberhard2021, FJSS21} seem to indicate that the eigenvalues of random matrices over $\F_q$ also exhibit universal behaviour. It would be very interesting to see if there is a deeper connection, as well as if other examples of universal behaviour for random objects defined over finite fields could be found.

\subsection{Statistics for the uniform model}
Uniformly random polynomials (or equivalently monic polynomials) have been intensely studied in both the probability and combinatorics literature. Some of the strongest results are due to Arratia, Barbour and Tavar\'e \cite{ABT93}, who established strong, quantitative approximation results for not only the low degree factors, but also the high degree ones. They also established a functional central limit theorem and Poisson-Dirichlet limit theorems. Their results fit into the subsequent general theory of logarithmic combinatorial structures, see \cite{ABT03}. However, their methods rely heavily on the fact that the polynomials are uniformly sampled, and except for certain special deformations analogous to the Ewen's sampling formula for random permutations, their results do not extend to more general models of random polynomials.

Our results show that the low degree irreducible factors of $f$ behave very similarly to those for a uniformly random monic polynomial. This is useful because the distribution of $N_i(\overline{f})$ and $N_i'(\overline{f})$ are very well understood, and this knowledge can be transferred to say something about $N_i(f)$ and $N_i'(f)$. We briefly survey some results on low degree factors for the uniform model $\overline{f}$. We note that in this section only, $N_1(f)$ and $N'_1(f)$ include the irreducible factor $x$, which we exclude for the rest of the paper.

\begin{theorem}[{\hspace{1sp}\cite[Theorem 3.1, Corollary 3.3]{ABT93}}]
\label{thm: uniform low deg}
Let $\pi(i)$ denote the number of monic irreducible polynomials of degree $i$ in $\F_q[x]$. For $i\geq 1$, let $X_i$ and $Y_i$ be independent binomial and negative binomial random variables of parameters $\pi(i)$ and $q^{-i}$. Then
\begin{equation*}
    d_{TV}\left((N_i(\overline{f}))_{i\leq N},(X_i)_{i\leq N}\right)=O\left(N\exp\left(-\frac{n}{2N}\log\frac{4}{3}\right)\right)
\end{equation*}
and
\begin{equation*}
    d_{TV}\left((N_i'(\overline{f}))_{i\leq N},(Y_i)_{i\leq N}\right)=O\left(N\exp\left(-\frac{n}{2N}\log\frac{4}{3}\right)\right).
\end{equation*}
\end{theorem}
Note that when $q$ is large, $\pi(i)\sim \frac{q^i}{i}$, and so both $X_i$ and $Y_i$ become close to Poisson random variables of mean $\frac{1}{i}$. In fact, when $q$ is large, the $N'_i(\overline{f})$ (and also the $N_i(\overline{f})$) are extremely close to the distribution of cycles in a random permutation.
\begin{theorem}[{\hspace{1sp}\cite[Corollary 5.6]{ABT93}}]
\label{thm: rand poly and rand perm}
Let $C_i$ denote the number of cycles of size $i$ in a uniformly random permutation in $S_n$. Then
\begin{equation*}
    d_{TV}\left((N'_i(\overline{f}))_{i=1}^n,(C_i)_{i=1}^n\right)\leq q^{-1}+O\left(q^{-\frac{3}{2}}\right).
\end{equation*}
\end{theorem}
In fact, a lower bound of order $q^{-1}$ is also known \cite{ABT93}, and this theorem was given another proof in \cite{BSG18}.

Our results imply that all these statements hold for $N_i(f)$ and $N_i'(f)$ as well, albeit for a much lower upper bound on the highest degree $i$.

Actually Theorem \ref{thm: uniform low deg} follows from the following more general result on the joint distributions of all multiplicities of irreducible factors up to a given degree.
\begin{theorem}[{\hspace{1sp}\cite[Theorem 3.1]{ABT93}}]
\label{thm: unif root TV}
Let $N_\phi(\overline{f})$ denote the multiplicity of an irreducible factor $\phi$ in $\overline{f}$. Let $Z_\phi$ denote independent geometric random variables of parameter $q^{-\deg(\phi)}$. Then for all $N\geq 1$,
\begin{equation*}
    d_{TV}\left((N_\phi(\overline{f}))_{\deg(\phi)\leq N},(Z_\phi)_{\deg(\phi)\leq N}\right)=O\left(N\exp\left(-\frac{n}{2N}\log\frac{4}{3}\right)\right).
\end{equation*}
\end{theorem}

Finally, we also mention some results on the high degree factors. Recall that the \emph{Poisson-Dirichlet} process is a random variable taking values in infinite vectors $(x_i)$ with $x_1\geq x_2\geq \dotsc$ and $\sum x_i=1$. It has a natural description in terms of the following stick-breaking process. Let $U_i$ be independent uniform random variables on $[0,1]$, and let $V_i=U_i\prod_{j<i} (1-U_j)$. One can think of the $V_i$ as sampled by placing points inductively, uniformly on the rightmost interval. Then the Poisson-Dirichlet process is given by the sorted lengths of the intervals defined by the points $V_i$.

\begin{theorem}[{\hspace{1sp}\cite[Remark 5.13]{ABT93}}]
Let $L=(L_i(\overline{f})/n)$ denote the normalized degrees of the irreducible factors of $\overline{f}$ a uniformly chosen monic polynomial of degree $n$, in descending order. Then $L$ converges to the Poisson-Dirichlet process.
\end{theorem}

Much is also known about the total number of factors. It's known that the number of factors is close to a Poisson of mean $H_n$, where $H_n$ is the $n$th harmonic number. 
\begin{theorem}[{\hspace{1sp}\cite[Theorem 6.8]{ABT93}}]
\label{thm: total number unif}
Let $N'(\overline{f})$ denote the total number of irreducible factors for a random monic polynomial of degree $n$. Let $Z$ denote a Poisson random variable of mean $H_n$. Then
\begin{equation*}
    d_{TV}(N'(\overline{f}), Z)=O\left(\log^{-\frac{1}{2}}n\right).
\end{equation*}
\end{theorem}
There is also recent work of Elboim and Gorodetsky, who obtained optimal total variation bounds to the number of cycles in a random permutation \cite{EG20}, which in particular implies Theorem \ref{thm: total number unif}.

\subsection{Main results}
Our main results establish a form of universality for the number of low degree irreducible factors of random polynomials over finite fields. Since our strongest results require some technical assumptions and are only effective for certain ranges of the parameters we consider, we have left their statements for Section \ref{sec: main results}. Instead, we state some simpler consequences which apply with only mild restrictions.

Let $f\in \F_p[x]$ be the random polynomial defined by
\begin{equation*}
    f(x)=\sum_{i=0}^n \varepsilon_ix^i,
\end{equation*}
with the $\varepsilon_i$ independently drawn from some distribution $\mu$ on $\F_p$. Let $\eta=1-\max_{x\in\F_p}\mu(x)$. This parameter has previously appeared in the study of universality for random matrices over $\F_p$ (see e.g. \cite{LMN21}), and $\eta>0$ ensures that $\mu$ is not concentrated at a single point. Let $\overline{f}(x)$ be a uniformly random monic polynomial in $\F_p[x]$ of degree $n$. For any polynomial $g$, let $N_i(g)$ denote the number of distinct irreducible factors of $g$ of degree $i$, and let $N_i'(g)$ denote the total number of irreducible factors of $g$ of degree $i$, counted with multiplicity. We note that $N_1(g)$ and $N'_1(g)$ exclude factors of $x$, since the number of these clearly depends on the distribution $\mu$.

Our most general results apply even when $p$ grows rapidly with $n$, but require the coefficients to lie in $\F_p$ for a prime $p$. We first consider the joint distribution of the number of distinct irreducible factors of degrees not too large in $n$.
\begin{corollary}
\label{cor: no mult}
Suppose that $p\leq e^{n^{\frac{1}{8}}}$. Then for any $\delta>0$,
\begin{equation*}
    d_{TV}\left((N_i(f))_{i\leq n^{\frac{1}{8}-\delta}},(N_i(\overline{f}))_{i\leq n^{\frac{1}{8}-\delta}}\right)=O(n^{-\delta}),
\end{equation*}
where the implicit constant depends only on $\eta$ and $\delta$.
\end{corollary}
We also consider the number of irreducible factors counted with multiplicity.
\begin{corollary}
\label{cor: with mult}
Suppose that $p\leq e^{n^{\frac{1}{13}}}$. Then for any $\delta>0$,
\begin{equation*}
    d_{TV}\left((N'_i(f))_{i\leq n^{\frac{1}{13}-\delta}},(N'_i(\overline{f}))_{i\leq n^{\frac{1}{13}-\delta}}\right)=O(n^{-3\delta}),
\end{equation*}
where the implicit constant depends only on $\eta$ and $\delta$.
\end{corollary}

Together with Theorem \ref{thm: uniform low deg}, this immediately implies the following limit theorems.
\begin{corollary}
\label{cor: universal limit}
Let $X_i$ and $Y_i$ be independent binomial and negative binomial random variables of parameters $\pi(i)$ and $p^{-i}$. For any $\delta$, if $p\leq e^{n^\frac{1}{8}}$, then
\begin{equation*}
    d_{TV}\left((N_i(f))_{2\leq i\leq n^{\frac{1}{8}-\delta}},(X_i)_{2\leq i\leq n^{\frac{1}{8}-\delta}}\right)=O(n^{-\delta}),
\end{equation*}
and if $p\leq e^{n^\frac{1}{13}}$, then
\begin{equation*}
    d_{TV}\left((N'_i(f))_{2\leq i\leq n^{\frac{1}{13}-\delta}},(Y_i)_{2\leq i\leq n^{\frac{1}{13}-\delta}}\right)=O(n^{-3\delta}).
\end{equation*}
\end{corollary}

\begin{remark}
We note that it is necessary to exclude the irreducible factor $x$, because whether it divides $f$ and with what multiplicity is easily seen to be dependent on $\mu(0)$. In particular, the multiplicity of $x$ as a factor is a truncated geometric random variable of parameter $\mu(0)$. We will ignore factors of $x$, and it is safe to simply assume that $\mu(0)=0$, although this is not necessary for our results to hold. This essentially amounts to randomizing the degree and conditioning on the constant term being non-zero, which would ultimately not affect our arguments.

This is also why Corollary \ref{cor: universal limit} is stated with $i\geq 2$. For $i=1$, one needs to specifically remove $0$, and so the limiting distribution should be binomial and negative binomial with parameters $\pi(1)-1=p-1$ and $p^{-1}$. One could use the same arguments as in \cite{ABT93} to include $i=1$ as well.
\end{remark}

\begin{remark}
There are many models we could draw $\overline{f}$ from, whether uniformly from all polynomials of degree $n$, those of degree at most $n$, or monic polynomials of degree $n$. It makes no difference in our analysis, since the small irreducible factors have basically the same distribution, and indeed our proof uses a moment matching argument which only sees low-order moments, which are nearly identical for all these models. This is easy to see, as the roots do not depend on whether the polynomial is taken to be monic or not, and a uniform polynomial of degree at most $n$ is exponentially likely to have a large degree, and so will have statistics close to that of a random monic polynomial.

For this paper, we will always work with monic polynomials, as they are a bit easier to work with and most results in the literature are on this model.
\end{remark}

\begin{remark}
These results are stated to maximize the degree of the irreducible factors considered. We can take $p$ up to $e^{n^{\frac{1}{4}-\delta}}$ for any
$\delta>0$ at the cost of considering lower degree irreducible factors.
\end{remark}

For a single degree, we can show that the number of irreducible factors of a fixed degree $i$ of $f$ and of the uniform model $\overline{f}$ have approximately the same distribution for $i$ up to $n^{1/2-\epsilon}$. 

\begin{corollary}\label{cor: min}
Suppose that $i \le n^{1/2-2\epsilon}$ and $p < \exp(n^{1/2-\epsilon}/i)$. Then
\[
d_{TV}(N_i(f),N_i(\overline{f})) = O_{\epsilon,\eta}(\exp(-n^{c\epsilon})+\min(\exp(-ci),n^{-c})),
\]
where $c$ is an absolute constant independent of $\eta$ and $\epsilon$. 
\end{corollary}

We obtain similar conclusions for linear statistics of the number of irreducible factors of degree up to $n^{1/2-\epsilon}$, for example, the total number of irreducible factors of degree at most $n^{1/2-\epsilon}$. Similar results could also be derived for a bounded or slowly growing number of such statistics. We remark that the threshold $n^{1/2}$ seems to be the fundamental limit of our technique, and it would be very interesting to derive universality results beyond this threshold.

If we fix the finite field, we can in fact handle coefficients in an arbitrary finite field $\F_q$. We now take $\mu$ to be a distribution on $\F_q$, and let $\eta=1-\max_{V\subseteq \F_q}\mu(V)$ where the maximum is taken over all proper $\F_p$ affine subspaces $V$, and let $f$ and $\overline{f}$ be defined as above, but with coefficients in $\F_q$.

\begin{corollary}
\label{cor: with q}
Fix a prime power $q$. Then for some small constant $c$ depending only on $\eta$ and $q$, if $N=cn^{\frac{1}{4}}/\log^{\frac{1}{2}}n$, then
\begin{equation*}
    d_{TV}\left((N_i(f))_{i\leq N},(N_i(\overline{f}))_{i\leq N}\right)=O(e^{-n^{\frac{1}{4}}})
\end{equation*}
and if $N=cn^{\frac{1}{5}}/\log^{\frac{4}{5}}n$, then
\begin{equation*}
    d_{TV}\left((N'_i(f))_{i\leq N},(N'_i(\overline{f}))_{i\leq N}\right)=O(e^{-n^{\frac{1}{5}}}),
\end{equation*}
where the constants depends only on $\eta$ and $q$.
\end{corollary}
Again, together with Theorem \ref{thm: uniform low deg} this implies that the $N_i(f)$ and $N_i'(f)$ converge to independent binomial and negative binomial random variables.

All of the results in this section follow from stronger bounds which are stated in Section \ref{sec: main results}, and their proofs are given there as well.

\begin{remark}
While we focus on a few statistics in this paper, mainly the statistics $N_i(f)$ and $N_i'(f)$, in fact our techniques can handle any reasonable function of the irreducible factors of $f$ and their multiplicity, subject to the restrictions that the irreducible factors have low degree, and that the number of statistics considered is small. We did not find a way to make this precise, but we note that other statistics of interest such as the joint multiplicities of a small number of possible irreducible factors, or the number of irreducible factors of degree $i$ and multiplicity $m$, could be studied with our techniques, and quantitative bounds could be obtained.

In particular, using our methods and Theorem \ref{thm: unif root TV}, one could bound the distance between the joint multiplicities of irreducible polynomials $\phi_1,\dotsc, \phi_k$ and geometric random variables of parameters $q^{-\deg(\phi_1)},\dotsc, q^{-\deg(\phi_k)}$.
\end{remark}

\subsection{Further questions}
Given the wealth of knowledge on the irreducible factors of uniform random polynomials, it is natural to wonder to what extent these distributions are universal. 

\begin{question}
 To what extent are the statistics of the irreducible factors of random polynomials in $\F_q[x]$ universal? That is, let $f=\sum \varepsilon_i x^i$ be a random polynomial with independent and identically distributed coefficients in $\F_q$. What statistics are close to that of a uniformly chosen monic polynomial?
\end{question}

Our results answer this question for low degree factors. Our methods are not suitable for studying the high degree irreducible factors. On the other hand, numerical simulations suggest that even the high degree irreducible factors exhibit universality. Based on this, we make the following conjecture on the maximal degree of an irreducible factor.

\begin{conjecture}
\label{conj: max deg}
Let $f=\sum\varepsilon_i x^i$ be a random polynomial with independent coefficients in $\F_p$. The maximal degree of an irreducible factor of $f$, normalized by the total degree, converges to the maximum of a Poisson-Dirichlet process.
\end{conjecture}
We give some evidence in Section \ref{sec: numerics}. This would immediately follow from the following stronger conjecture, which we do not have any additional evidence for.
\begin{conjecture}
Let $f=\sum\varepsilon_i x^i$ be a random polynomial with independent coefficients in $\F_p$. Then the normalized degrees of the irreducible factors converge to a Poisson-Dirichlet process.
\end{conjecture}
In addition, one could ask similar questions about the total number of factors, the medium-degree factors, and so on. We leave all of these questions as open problems.

\subsection{Proof idea}
We now give a heuristic explanation for the bounds we obtain, and an idea of their proof.

The starting point is that the values of the random polynomial $f(x)=\sum_{i=0}^n \varepsilon_i x^i$ at an element $\alpha\in \F_{p^e}$ has the same distribution as the states at time $n$ of the Markov chains defined by $X_t=\alpha X_{t-1}+\varepsilon_t$. It is not hard to show that these Markov chains converge to the uniform distribution, and in fact recent work of Breuillard and Varj\'u \cite{BV19poly,BV19walk} show that if $q=p$ is prime, for most $\alpha$ and $p$, this Markov chain converges quickly. Note that their work cannot be applied to all $\alpha$, and so a major difficulty in our work is to show that there are not so many exceptional $\alpha$, and that they can be dealt with separately. Once the Markov chain equidistributes, we can immediately conclude that $\alpha$ is a root of $f(x)$ with probability $q^{-1}$, matching the probability for the uniform model $\overline{f}(x)$. 

In fact, while we ultimately avoid studying these random walks, we use the same techniques, extending them to handle the joint distribution of $f(\alpha_i)$ at multiple roots $\alpha_i$, along with the derivatives of $f$. We ultimately obtain two bounds on the Fourier coefficients for the distribution of the $f(\alpha_i)$, given by Proposition \ref{prop: weak fourier bound} and Theorem \ref{thm: konyagin fourier bound}.

The first bound we obtain, Proposition \ref{prop: weak fourier bound}, morally comes from the fact that these Markov chains converge to stationarity in at most order $q^2$ steps. This already allows us to handle the case of constant $q$, and is effective for small $q$. 

The other bound, Theorem \ref{thm: konyagin fourier bound} uses a bound originally due to Konyagin, see \cite{BV19poly,BV19walk}, and so morally comes from the fact that if $q=p$ is prime, for most $\alpha$, these Markov chains converge to stationarity in order $\log^2 p\log^5\log p$ steps. We cannot use the more precise techniques that give the sharper bounds obtained in \cite{BV19walk,BV19poly,EV21}, as these do not extend as readily to handle multiple roots or derivatives. In particular, when considering derivatives, the measures one considers lack the self-similarity property that was crucial in the analysis done in \cite{BV19walk}, but Konyagin's argument still works.

This lets us handle roots of high multiplicative order. We then handle the remaining roots by showing that they do not appear with high probability, and this requires that $p$ is large enough. This method is effective up to $p\leq e^{n^c}$ for reasonable $c>0$, but gives a polynomial rather than exponential error bound.

The second bound, Theorem \ref{thm: konyagin fourier bound}, is more involved, and relies on an argument originally due to Konyagin \cite{K92} which shows that the $\alpha$ for which the Markov chain above mixes slowly must have low multiplicative order. We extend this argument to handle multiple the values of $f$ and its derivatives at multiple roots, giving us information on their joint distribution. This allows us to effectively approximate the joint moments of the $N_i(f)$ and $N_i'(f)$ by the same moments for the uniform model.

Once we have these moment estimates, we then use these to obtain a bound on the total variation distance via the following proposition, which we could not find in the literature (although similar ideas have appeared, see e.g. \cite{J97}) and may be of independent interest.

\begin{proposition}
\label{prop: total variation bound}
Let $Z=(Z_1,\dotsc, Z_N)$ and $Z'=(Z_1',\dotsc, Z_N')$ be two integer-valued random vectors. Fix $H\in\N$ and let
\begin{equation}
\label{eq: moment error assumption}
    \varepsilon=\sup_{\sum k_i\leq H}\left|\E\left(\prod_{i=1}^{N} Z_i^{k_i}\right)-\E\left(\prod_{i=1}^{N} Z_i'^{k_i}\right)\right|.
\end{equation}
Suppose that
\begin{equation}
\label{eq: moment tail assumption}
    \E\left(\sum_{i=1}^{N} |Z_i|\right)^H,\E\left(\sum_{i=1}^{N} |Z'_i|\right)^H\leq C.
\end{equation}
Then for all $a\in\Z^N$,
\begin{equation*}
    |\P[Z=a]-\P[Z'=a]|\leq N^{H-1}e^{\pi}\varepsilon+2\frac{C\pi^H}{H!}.
\end{equation*}
\end{proposition}
\begin{remark}
Note that Proposition \ref{prop: total variation bound} must be summed over the support of $Z$ and $Z'$ to obtain a total variation bound, so in practice $\varepsilon$ must be very small or one needs good tail bounds for this bound to be useful. The easiest case is when $\varepsilon=0$ and the moments match exactly. In our application, $\varepsilon$ will be exponentially small, and so we must carefully pick our parameters to ensure that the bound is effective.
This result can also be sharpened in the case when the errors are not uniformly small.
\end{remark}

\begin{remark}
We do not explicitly make the connection, but our results imply mixing time bounds of order $\log^2 p \log^5\log p$ for certain higher-dimensional analogues of the $ax+b$ process. These are weaker than the expected order $\log p\log\log p$ mixing time, but this improvement would not give very much in our setting, improving the dependence on $p$ in the upper bounds. Recent work of Dubail and Massouli\'e \cite{DM21} establishes this bound for a large class of higher-dimensional analogues, but unfortunately their results do not apply to our setting, and their bounds do not give a uniform control which is necessary for our applications. Nevertheless, it would be interesting to see if these ideas, or the more refined ideas giving the $\log p\log\log p$ mixing time bound for the $1$-dimensional $ax+b$ chain, could be used in our setting.
\end{remark}

\subsection{Related work}
Our work seems morally related to results on universality for random polynomials over $\R$ or $\C$. Here, the comparison is with a random polynomial with Gaussian coefficients. We do not attempt to review all the literature, and refer the reader to the cited papers for further background and references. The limiting density for roots \cite{IZ97} and local correlations \cite{TV15,DNV18} are known to exhibit universal behaviour. More recently, it has been established that the moduli of the roots converge to a Poisson point process \cite{cook2021universality2}, again matching the behaviour of the Gaussian model \cite{MS20}.  It is interesting to note that \cite{CN21, cook2021universality2} use similar techniques, studying random walks related to their random polynomial and their derivatives, as well as using Fourier-theoretic arguments.

There has also been some recent progress on the roots of random $p$-adic polynomials. Recent work of Shmueli \cite{S21} found the expected number of roots for a random polynomial over $\Q_p$ whose coefficients are randomly drawn from $\Z_p$, but are not necessarily Haar-distributed. There are some striking similarities with the finite field case that we study, and it would be interesting to see if some of our results and methods could be used to approach this problem as well.

Random matrices over finite fields have also been quite intensely studied recently, and many properties of uniform random matrices over finite fields are now known to be universal. We will only survey some recent results, and refer the reader to \cite{LMN21} for more references.

Recent work of Luh, Meehan and Nguyen \cite{LMN21} show that the rank distribution and the small factors of the characteristic polynomial of a random matrix over $\F_p$ are universal, at least when $p$ is fixed. The analogous results \cite{S88, NP98} for the uniform model were known much earlier. Eberhard \cite{eberhard2021} strengthened the error bounds and extended some of these results to $\F_q$, and Ferber, Jain, Sah and Sawhney \cite{FJSS21} showed similar results for the rank distribution of symmetric matrices. These latter two works were motivated by the study of random $\pm 1$ matrices over $\Z$.

Finally, the distribution of factors for uniformly random monic polynomials is exactly equal to the distribution of cycles in a deck of cards after a $q$-shuffle (a generalization of a riffle shuffle) \cite{DMP95}, and this has been extended to other Coxeter groups \cite{F00,F01}. It's unclear how to interpret non-uniform random polynomials via card shuffling, but it would be interesting if a connection could be made. It would also be interesting to see if $a$-shuffles had some universal behavior.

\subsection{Outline}
The rest of the paper is structured as follows. In Section \ref{sec: prelim}, we set some notation and recall basic facts about Mahler measure and the Hasse derivative. In Section \ref{sec: low order}, we establish similar bounds when $q$ is small, and also establish a Hal\'asz-type bound for the probability that a given polynomial divides $f$. In Section \ref{sec: konyagin}, we establish bounds on Fourier coefficients effective when $q=p$ is prime and large, adapting an argument originally due to Konyagin \cite{K92}. In Section \ref{sec: tv bound}, we prove Proposition \ref{prop: total variation bound} and give moment bounds for the uniform model. In Section \ref{sec: moments}, we show that the joint moments of the $N_i$ and $N_i'$ for $f$ and $\overline{f}$ are close. In Section \ref{sec: main results}, we state and prove our strongest results, and prove the corollaries stated in the introduction. Finally, in Section \ref{sec: numerics}, we present some numerical simulations which support our results and suggest some interesting directions for further study.

\section{Preliminaries}
\label{sec: prelim}
In this section, we set notation and review some basic facts about Mahler measure and Hasse derivatives. The reader may safely skip this section and refer back when needed.

\subsection{Notation}
Throughout, $\Tr$ will denote the field trace of a finite extension $\F_{q^e}/\F_q$, with the extension clear from context. We define $e_p(x)=\exp\left(\frac{2\pi i x}{p}\right)$. We will let $C$ and $c$ denote a large and small positive constant respectively, that may change from line to line.
We use $f\ll g$ or $f =O(g)$ to denote that there exists a positive
constant  $C$ such that $f \le C g$.

\subsection{Assumptions on parameters}
Throughout this paper, we will let $f$ denote a random polynomial with coefficients drawn independently from a distribution $\mu$. We define parameters $n$, $N$, $H$, $K$. We will let $n$ denote the degree of $f$, $N$ the degree of the largest degree roots we wish to study, $H$ the largest number of distinct roots we wish to study, and $K$ the highest order Hasse derivative we wish to study. We will eventually assume that
\begin{equation}
\label{eq: assumptions}
    N=n^c,\qquad H=N\log n,\qquad K=N\log^2 n,\qquad p\leq e^{n^c}
\end{equation}
with $c>0$ some explicit constants that will depend on whether we count irreducible factors with or without multiplicity. We will only need to take $K>0$ when we study irreducible factors with multiplicity.

\subsection{Mahler measure}
The \emph{Mahler measure} of a polynomial $f(x)\in \C[x]$ with $f(x)=c_d\prod_{i=1}^{d} (x-\alpha_i)$, denoted $M(f)$, is defined by
\begin{equation*}
    M(f)=|c_d|\prod_{i=1}^{d} \max(1,|\alpha_i|)=\exp\left(\frac{1}{2\pi}\int_0^{2\pi}\log|f(e^{i\theta})|d\theta\right).
\end{equation*}
It was previously used by Breuillard and Varj\'u in their study of certain random walks over finite fields \cite{BV19poly,BV19walk} and we borrow their ideas heavily. See Section 1.6 of \cite{BG06} for some further background.

Mahler measure is multiplicative, $M(fg)=M(f)M(g)$, and if $f(x)=\sum_{i=0}^{d} c_i x^i$, there is also an upper bound
\begin{equation*}
    M(f)\leq \sqrt{\sum_{0\le i\le d} c_i^2}.
\end{equation*}
The key property of Mahler measure we exploit is that it provides a way to detect if $f$ is a cyclotomic polynomial (or $f(x)=x$) when $f$ is a monic irreducible integer polynomial. Specifically, if $f\in \Z[x]$ is monic and irreducible of degree $d$, then by Dobrowolski's bound \cite{D79}, either
\begin{equation*}
    M(f)\geq 1+c\left(\frac{\log \log d}{\log d}\right)^3,
\end{equation*}
or $f(x)$ is either a cyclotomic polynomial, or $f(x)=x$. If $f$ is a cyclotomic polynomial, then we can obtain useful bounds on its multiplicative order.

\subsection{The Hasse derivative}
We would like to study the multiplicity of roots of a polynomial $f$ through studying the roots of $f$ and its derivatives. However, since the $p$th derivative of any polynomial with coefficients in $\F_p$ is $0$, we use the following alternative. See Section 5.10 of \cite{HKT08} for further details.
\begin{definition}[Hasse derivative]
Let $f(x)=\sum c_i x^i\in R[x]$ be a polynomial over a ring $R$ (for us, $R=\Z$ or $R=\F_q$). We define the $k$th \emph{Hasse derivative} of $f$, denoted $D^{(k)}f(x)$, by
\begin{equation*}
    D^{(k)}f(x)=\sum c_i{i\choose k}x^{i-k},
\end{equation*}
with the understanding that if $k>i$ then ${i\choose k}=0$.
\end{definition}

We will only need the following basic but useful properties.

\begin{itemize}
    \item $D^{(k)}(fg)(x)=\sum_{i+j=k} D^{(i)}f(x)D^{(j)}g(x)$.
    \item  $f(x)=\sum_{k=0}^{\deg(f)} D^{(k)}f(\alpha)(x-\alpha)^k$.
    \item $f(x)$ has a root $\alpha$ of multiplicity $r$ if and only if $D^{(k)}f(\alpha)=0$ for all $k\leq r$.
\end{itemize}

If $\alpha$ generates an extension $\F_{p^e}$ of $\F_p$, then $1,\dotsc, \alpha^{e-1}$ forms a basis. The following lemma extends this to multiple roots $\alpha_j$, along with derivatives of the monomials.

\begin{lemma}
\label{lem: lin ind}
Let $\alpha_j\in \F_{q^{e_j}}$ for $j=1,\dotsc, n$, so that each $\alpha_j$ does not lie in any smaller subfield, and none of the $\alpha_j$ are Galois conjugates. For each $j$, let $k_j\in \N$. Let $d=\sum_{1\le j \le n} e_j(k_j+1)$. Let $m\geq 0$.

The vectors $(\alpha_j^{i-k}{i\choose k})\in\prod _j \prod _{k\leq k_j}\F_{q^{e_j}}$ for $i=m,\dotsc, m+d-1$ form a basis.
\end{lemma}
\begin{proof}
Suppose that $\sum_{i=m}^{m+d-1} c_i \alpha_j^{i-k}{i\choose k}=0$ for all $k\leq k_j$ and $j\leq n$. Then we have that the polynomial $f(x)=\sum_{i=m}^{m+d-1} c_i x^i$ satisfies $D^{(k)}f(\alpha_j)=0$ for all $k\leq k_j$ and $j\le n$. This implies in particular that each $\alpha_j$ is a root of $f$ of multiplicity at least $k_j+1$. As the $\alpha_j$ are not Galois conjugates, we show this forces $f(x)=0$. Since any non-zero polynomial having all the $\alpha_j$ roots of multiplicity $k_j+1$ respectively must be divisible by the minimal polynomials of the $\alpha_j$ at least $k_j+1$ times, such a polynomial must have a factor of degree at least $\sum e_j(k_j+1)=d$ corresponding to these roots. Note that $f$ has degree at most $m+d-1$, and is divisible by $x^m$, which implies that $f(x)=0$, and so $c_i=0$ for all $i$. 
\end{proof}

\section{Uniform estimates}
\label{sec: low order}
\subsection{General setup}
\label{sec: setup}
Throughout the next few sections, we fix the following notation and general setup. Fix some prime $p$, $q=p^e$ for some positive integer $e$. Let $H\in\N$ and $\alpha_i\in \F_{q^{e_i}}$ for $i\leq H$, with the $\alpha_i$ not lying in a smaller subfield, and none a Galois conjugate of another. These will be the roots we will consider. Let $\mathcal{K}_i\subseteq \N$ for $i\leq H$, which will be the set of derivatives we consider for $\alpha_i$, let $K_i=\max \mathcal{K}_i$, and let $k^*=\max_i K_i+1$ denote the largest derivative considered. Let $d=\sum_{1\le i \le H} e_i(K_i+1)$. Let $\beta_{i,k}\in\F_{q^{e_i}}$ for $k\in \mathcal{K}_i$ and $i\leq H$, with all $\beta_{i,k}$ non-zero. Let $V = \prod _i \left(\F_{q^{e_i}}\right)^{\mathcal{K}_i}$. 

Let $X_i$ be a sequence of independent and identically distributed $\F_q$-valued random variables, and let $\mu$ denote its distribution. Let $\eta = 1 - \max_{V\subseteq \F_q} \mu(V)$, where $V$ is taken over all affine $\F_p$ subspaces. We wish to study the distribution of $\sum_{i=0}^n X_i D^{(k)}(\alpha_j^i)$ for $j\leq H$ and $k\in \mathcal{K}_j$ as a random variable in $V$, and we let $\nu_n$ denote its distribution.

In this section, we obtain some uniform estimates that hold for any roots $\alpha$. While these are not strong enough when $p$ is large, they are necessary to handle certain low-order roots that cannot be handled in any other way.

\subsection{Fourier bounds}
We first obtain uniform bounds on the Fourier coefficients of $\nu_n$.
\begin{lemma}
\label{lem: n>p^2 bound}
Let $d_1,d_2\in\N$, and let $T:\F_q^{d_1}\to\F_q^{d_2}$ be a surjective linear map. Let $\mu$ be a probability measure on $\F_q$ whose support generates $\F_q$. Suppose that for each $\beta\in\F_q$, the fraction of $i\in [d_1]$ such that $\beta\cdot Tv_i=0$ is at most $1-\gamma$ if $\beta\neq 0$, where $v_i$ is the standard basis for $\F_q^{d_1}$. Define $\eta$ to be the minimum probability of $\mu$ in the complement of a proper $\F_p$ affine subspace of $\F_q$. Then for $\beta\neq 0$,
\begin{equation*}
    \widehat{T_*\mu^{\otimes d_1}}(\beta)\leq (1-\eta/p^2)^{\gamma d_1},
\end{equation*}
where $T_*\mu^{\otimes d_1}$ denotes the pushforward under $T$ of the $d_1$-fold product measure of $\mu$ on $\F_q^{d_1}$.
\end{lemma}
\begin{proof}
We have
\begin{equation*}
\begin{split}
    \widehat{T_*\mu^{\otimes d_1}}(\beta)&=\sum_{x\in \F_p^{d_1}}\left(\prod_{i=1}^{d_1} \mu(x_i)\right)e_p(\Tr(\beta\cdot T(x)))
    \\&=\sum_{x\in \F_p^{d_1}}\prod_{i=1}^{d_1} \mu(x_i)e_p(\Tr(x_i\beta\cdot T(v_i)))
    \\&=\prod _{i=1}^{d_1} \widehat{\mu}(\beta\cdot Tv_i).
\end{split}
\end{equation*}
By the definition of $\eta$, $|\widehat{\mu}(\beta\cdot Tv_i)|\leq |1-\eta+\eta e^{2\pi i/p}| \le 1-\eta/p^2$ if $\beta \cdot Tv_i\neq 0$. By assumption, this happens at least $\gamma d_1$ times.
\end{proof}

Combining Lemma \ref{lem: lin ind} and Lemma \ref{lem: n>p^2 bound}, we obtain the following bounds.

\begin{proposition}
\label{prop: weak fourier bound}
With the notation of Section \ref{sec: setup}, we have
\begin{equation*}
    |\widehat{\nu_n}(\beta)|\leq e^{-\frac{\eta n}{dp^2}}
\end{equation*}
where $\eta = 1-\max_{V\subsetneq \F_q}\mu(V)$. 
\end{proposition}
\begin{proof}
We let $T:\F_p^n\to \F_p^{d'}$ (where $d'=\sum e_i|\mathcal{K}_i|$) be the map taking $(x_0,\dotsc, x_n)$ to the vector $(D^{(k)}f(\alpha_i))_{i\leq H, k\in \mathcal{K}_i}$, where $f(\alpha)=\sum x_i \alpha^i$. By Lemma \ref{lem: lin ind}, this map is surjective, and moreover, $Tv_i$ for $i=j,j+1,\dotsc, j+d-1$ spans $\F_p^{d'}$. Thus for each non-zero $\beta$, there is at least one $i$ for which $\beta\cdot Tv_i\neq 0$, and so Lemma \ref{lem: n>p^2 bound} gives the desired bound.
\end{proof}

\subsection{Hal\'asz-type bounds}
Next, we establish the following Hal\'asz-type bound for the probability that a random polynomial $f(x)$ has some collection of roots with given multiplicities. While we only need the case of a single root without multiplicity, we believe the stronger result may be of some independent interest. We remark that there is an extensive literature on anti-concentration type bounds, and we refer readers to the survey \cite{NV}. In the setting of torsion abelian groups, Hal\'asz-type bounds and Littlewood-Offord inverse results are obtained in \cite{koenig2021note}. In our case, we take direct advantage of properties of the power sequence to obtain the desired bound in Proposition \ref{prop:halasz-bound}. We remark that one can also follow proof of typical Littlewood-Offord results \cite{NV} combined with inverse results of Freiman-type in general abelian groups \cite{GR} to obtain an upper bound in Proposition \ref{prop:halasz-bound} of the form $p^{-d} + O_{\eta}(n^{-\lambda_d})$ where $\lambda_d \to \infty$. For us, the bound provided in Proposition \ref{prop:halasz-bound} with an exponential decay in $d$ is  more convenient to use and leads to better quantitative bounds in our settings.

\begin{proposition}\label{prop:halasz-bound}
Let $f(x)=\sum_{1\le i \le n} \varepsilon_i x^i$ be a random polynomial of degree $n$ in $\F_p[x]$, with $\varepsilon_i$ independent and distributed according to $\mu$. Let $\alpha_1,\dots,\alpha_H$ be so that $\alpha_j$ in $\F_{p^{e_j}}$ and $\alpha_j$ does not lie in any proper subfield of $\F_{p^{e_j}}$. Let $d=\sum_{j\le H}e_j( K_j+1)$. Then
\begin{equation*}
    \P\left[(x-\alpha_i)^{K_i}|f(x) \text{ }\forall i\right]\leq \left(\frac{1}{p}+C\eta^{-\frac{1}{2}}\left\lfloor\frac{n}{d}\right\rfloor^{-\frac{1}{2}}\right)^d.
\end{equation*}
for some absolute constant $C$.
\end{proposition}
\begin{proof}
For simplicity, we assume $n$ is a multiple of $d$. We have
\begin{equation*}
\begin{split}
        \P\left[(x-\alpha_i)^{K_i}|f(x) \text{ }\forall i\right]&=\frac{1}{p^d}\sum_{\beta}\prod_{i} \E\left[e_p\left(\sum_{j\le H,k_j\le  K_j}\Tr(\beta_{j,k_j} \varepsilon_i D^{(k_j)}(\alpha_j^i))\right)\right].
\end{split}
\end{equation*}
By H\"older's inequality,
\begin{equation*}
\begin{split}
    &\sum_{\beta}\prod_i \left|\E\left[e_p\left(\sum_{j\le H,k_j\le K_j}\Tr(\beta_{j,k_j} \varepsilon_i D^{(k_j)}(\alpha_j^i))\right)\right]\right|\\
    &\leq \prod_{i=1}^{n/d}\left(\sum_{\beta } \prod_{h=0}^{d-1}\left|\E\left[e_p\left(\sum_{j\le H,k_j\le K_j}\Tr(\beta_{j,k_j} \varepsilon_{di+h} D^{(k_j)}(\alpha_j^{di+h}))\right)\right]\right|^{\frac{n}{d}}\right)^{\frac{d}{n}}
    \\&\leq \max_i\sum_{\beta} \prod_{h=0}^{d-1}\left|\E\left[e_p\left(\varepsilon_{di+h}\Tr\left(\sum_{j\le H,k_j\le K_j}\beta_{j,k_j}  D^{(k_j)}(\alpha_j^{di+h})\right)\right)\right]\right|^{\frac{n}{d}}.
\end{split}
\end{equation*}
By Lemma \ref{lem: lin ind}, the $d$ vectors $(D^{(k_j)}(\alpha_j^{di+h}))_{j\leq H, k_j\leq K_j}$ for $0\leq h\leq d-1$ (whose components are indexed by $j$ and $k_j$) form a basis. Let $\gamma_h$ be a dual basis with respect to the non-degenerate pairing
\begin{equation*}
    (x,y)\mapsto \Tr\left(\sum_{j\leq H, k_j\leq K_j}x_{j,k_j}y_{j,k_j}\right)
\end{equation*}
for this basis. Then writing $\beta=\sum c_j\gamma_j$, we have
\begin{equation*}
\begin{split}
    &\frac{1}{p^d}\sum_{\beta} \prod_{h=0}^{d-1}\left|\E\left[e_p\left(\sum_{j\le H,k_j\le K_j}\Tr(\beta_{j,k_j} \varepsilon_{di+h} D^{(k_j)}(\alpha_j^{di+h}))\right)\right]\right|^{\frac{n}{d}}
    \\=&\prod _{h=0}^{d-1}\left(\frac{1}{p}\sum_{c_h\in\F_p} \left|\E\left[e_p\left(c_h\varepsilon_{di+h}\right)\right]\right|^{\frac{n}{d}}\right).
\end{split}
\end{equation*}
But each factor can be bounded by $p^{-1}+C\left(\frac{\eta n}{d}\right)^{-\frac{1}{2}}$ (the proof can be found in Lemma 2.4 of \cite{M13} for example), giving the desired bound. 
\end{proof}

\section{Konyagin's argument for derivatives}
\label{sec: konyagin}
Recall the assumptions and notation of Section \ref{sec: setup}. In this section, we assume that $q=p$ is prime. We will sometimes need to treat the $X_i$ as $\Z$-valued random variables, and so we may lift $\mu$ to a measure on $\Z$ supported on $[0,p-1]$ in the obvious way.

The goal of this section is to prove the following bound for the Fourier coefficients of $\nu_n$.
\begin{theorem}
\label{thm: konyagin fourier bound}
Let $\nu_n$ and $\alpha_i$ be defined as in Section \ref{sec: setup}. 
There exists $C>0$ such that if $n\geq Cd\log p\log(d\log p)$, then either 
\begin{equation*}
    |\widehat{\nu_n}(\beta)|\leq e^{\frac{\eta n}{Cd\log p\log^5(d\log p)}},
\end{equation*}
for all $\beta$ whose components are all non-zero, or all $\alpha_i$ have multiplicative order at most $Cd\log p\log(d\log p)$. 
\end{theorem}

\subsection{Bounds on Fourier coefficients}
Our proof of Theorem \ref{thm: konyagin fourier bound} closely follows the proof of Proposition 25 in \cite{BV19poly}, using an argument due to Konyagin \cite{K92}.

Define
\begin{equation*}
    S_n :=\sum_{i\le H} \sum _{k\in \mathcal{K}_j} \Tr\left(\beta_{i,k} \alpha_i^{n-k}{n\choose k}\right).
\end{equation*}
Recall that we have assumed $\beta_{j,k}\neq 0$ for all $j\leq H$ and $k\in\mathcal{K}_j$.

\begin{lemma}
\label{lem: non-degenerate eqs}
Suppose that $f(x)=\sum_{i=0}^n a_i x^i\in \F_p[x]$ is of degree $n$, and that $a_0S_i+\dotsc +a_nS_{i+n}=0$ for all $i=i_0,i_0+1, \dotsc, i_0+d-1$. Then $D^{(k)}f(\alpha_j)=0$ for all $j$ and $k\leq  K_j$. In particular, $a_0S_i+\dotsc +a_nS_{i+n}=0$ holds for all $i$.
\end{lemma}
\begin{proof}
The hypothesis is equivalent to the statement that 
\begin{equation*}
    \sum _j \sum _k \Tr\left(\beta_{j,k}D^{(k)}(f(x)x^i)(\alpha_j)\right)=0
\end{equation*}
for $i=i_0,\dotsc, i_0+d-1$. Expanding out the derivatives, we obtain
\begin{equation*}
    \sum _j \sum _k \sum_{l_1+l_2=k} \Tr\left(\beta_{j,k}D^{(l_1)}f(\alpha_j)\alpha_j^{i-l_2}{i\choose l_2}\right)=0.
\end{equation*}

The vectors $(\alpha_j^{m-k}{m\choose k})\in\prod _j \prod _{k\leq  K_j}\F_{p^{e_j}}$ for $m=i_0,\dotsc, i_0+d-1$ form a basis by Lemma \ref{lem: lin ind}.

But the expression above can be written as
\begin{equation*}
\begin{split}
    &\sum _j \sum _k \sum_{l_1+l_2=k} \Tr\left(\beta_{j,k}D^{(l_1)}f(\alpha_j)\alpha_j^{i-l_2}{i\choose l_2}\right)
    \\=&\sum _j \sum_{l_2} \Tr\left(\left(\sum_{k\geq l_2}\beta_{j,k}D^{(k-l_2)}f(\alpha_j)\right)\alpha_j^{i-l_2}{i\choose l_2}\right).
\end{split}
\end{equation*}
This implies that
\begin{equation*}
    \sum_{k\geq l_2}\beta_{j,k}D^{(k-l_2)}f(\alpha_j)=0
\end{equation*}
    for all $j$ and $l_2$. This is a system of $ K_j+1$ many variables (the $D^{(k)}f(\alpha_j)$), and it can be solved by back substitution since starting from $l_2= K_j$, there is at most one new variable in each new equation. It follows that $D^{(k)}f(\alpha_j)=0$ for all $j\le H$ and $k\leq  K_j$.
\end{proof}

For $X=(x_0,\dotsc, x_N)\in\Z^{N+1}$,
we write $\Lambda_{E}(X)$ for the set of polynomials $f(x)=\sum_{0\le i \le n} c_i x^i \in \Z[x]$, such that $c_0x_j +\dotsc +c_n x_{j+n}=0$ holds for all $j\leq N-E$. When $E=N$, we write $\Lambda(X) = \Lambda_{N}(X)$ denote the set of polynomials $f(x) \in \Lambda_{deg(f)} (X)$. We now recall two results from \cite{BV19poly} which we will need to use.

\begin{lemma}[{\hspace{1sp}\cite[Lemma 27]{BV19poly}}]
With the same notations as above, if $f_1,f_2\in \Lambda(X)$ and $\deg f_1+\deg f_2\leq N$, then $\gcd(f_1,f_2)\in \Lambda(X)$.
\end{lemma}

\begin{corollary}[{\hspace{1sp}\cite[Corollary 28]{BV19poly}}]
\label{cor: ideal gen}
Suppose that $\Lambda(X)$ contains a polynomial of degree at most $N/2$. Then there exists a unique up to $\pm$ polynomial $f_0\in\Lambda(X)$ of minimal degree with coprime coefficients, and $P\in\Lambda(X)$ if and only if $f_0|P$ for $P$ such that $\deg(P)\leq N-\deg(f_0)$.
\end{corollary}

We can now prove the key tool in the proof of Theorem \ref{thm: konyagin fourier bound}.
\begin{proposition}\label{prop:konyagin}
Let $\widehat{S}_n$ denote the representative of $S_n$ in $[-p/2,p/2]$.

Let $L\geq 200 d\log p\log (d\log p)$ Suppose that
\begin{equation*}
    \sum_{n=n_0}^{n_0+L} (\widehat{S}_n)^2\leq \frac{p^2}{8\log(4L)}.
\end{equation*}
Then for each $j$, there is a polynomial $f_j$ of degree at most $3d\log p$ and Mahler measure at most $(d\log p)^{30 d\log p/L}$ such that $f_j(\alpha_j)=0$.
\end{proposition}

\begin{proof}
Set $E=3\lceil d\log p\rceil$. We first show that there is a polynomial $f_1\neq 0$ of degree at most $E$ such that $f_1\in \Lambda_E(\{\widehat{S}_n\}_{n=n_0}^{n_0+L})$.

Let $X_0,\dotsc, X_E$ be a sequence of independent random variables uniform on $\{-1,1\}$. By Hoeffding's inequality,
\begin{equation*}
    \P\left[\left|\sum_{0\le i \le E} X_i \widehat{S}_{i+n}\right|\geq \frac{p}{2}\right]\leq \frac{1}{2L}
\end{equation*}
for any $n=n_0,\dotsc, n_0+L-E$. By a union bound,
\begin{equation*}
    \P\left[\left|\sum_{0\le i \le E}  X_i \widehat{S}_{i+n}\right|\geq \frac{p}{2}\text{ for some $n$}\right]\leq \frac{1}{2},
\end{equation*}
and so the set $\Omega$ of $x=(x_0,\dotsc, x_E)\in \{-1,1\}^{E+1}$ such that $\left|\sum x_i \widehat{S}_{i+n}\right|< \frac{p}{2}$ for all $n$ has size at least $2^E\geq p^d$. Then by the pigeonhole principle, there exists $x,y\in \Omega$, $x\neq y$, such that
\begin{equation*}
    \sum x_i \widehat{S}_{i+n}=\sum y_i \widehat{S}_{i+n} \pmod{p}
\end{equation*}
for $n=n_0,\dotsc, n_0+d-1$.

Let $a_i=(x_i-y_i)/2$. By Lemma \ref{lem: non-degenerate eqs}, we have $\sum a_i \widehat{S}_{i+n}=0$ for all $n$. Then $f_1(x)=\sum a_i x^i\in \Lambda_E(\{\widehat{S}_n\}_{n=0}^L)\subseteq \Lambda(\{\widehat{S}_n\}_{n=0}^{\lceil 2L/3\rceil})$. Since $\deg f_1\leq \lceil 2L/3\rceil/2$, by Corollary \ref{cor: ideal gen}, we have a polynomial $f_0\in \Lambda(\{\widehat{S}_n\}_{n=0}^{\lceil 2L/3\rceil})$ with relatively prime coefficients and of minimal degree, unique up to a sign.

By Lemma \ref{lem: non-degenerate eqs}, we have $f_0(\alpha_j)=0$ for all $j$, and so for each $\alpha_j$, there is some $f_j$ an irreducible (over $\Z$) factor of $f_0$ such that $f_j(\alpha_j)=0$.

Now fix $j$, and note that $\deg f_j\leq E$, so what remains is to show $M(f_j)\leq (d\log p)^{30 d\log p/L}$. Let $\gamma_i$ denote the roots of $f_j$. Let $s=L/6E$. Then by Lemma \ref{lem: no roots of unity}, there exists a prime $q\in (s,2s]$ such that $\gamma_k/\gamma_l$ is not a $q$th root of unity for all $k\neq l$. This implies that the $\beta_i^q$ are distinct.

By the same argument as above, we can find a polynomial $f_2(x)=g_2(x^q)$ with $g_2=\sum b_i x^i$ of degree at most $E$ having coefficients in $\{-1,0,1\}$, and such that $\sum b_i \widehat{S}_{iq+n}=0$ for all $n=0,\dotsc, L-Eq$. This implies that $f_2\in \Lambda_{Eq}(\{\widehat{S_n}\}_{n=0}^L\subseteq \Lambda(\{\widehat{S}_{n}\}_{n=0}^{\lceil 2L/3\rceil})$ as $Eq\leq L/3$. By Corollary \ref{cor: ideal gen}, as $\lceil 2L/3\rceil\geq 2Eq$, we have $f_0$ divides $f_2$. Then the $\beta_i^q$ are all roots of $g_2$, and so $M(f_j)\leq M(g_2)^{\frac{1}{q}}\leq (E+1)^{\frac{1}{2q}}$ using $M(g_2)\leq \sqrt{\sum b_i^2}$. Since $q\geq L/6E$, we have $M(f_j)\leq (d\log p)^{30d\log p/L}$.
\end{proof}

Finally, we need the following lemma to prove Theorem \ref{thm: konyagin fourier bound}.
\begin{lemma}[{\hspace{1sp}\cite[Lemma 26]{BV19poly}}]
\label{lem: no roots of unity}
Let $a_1,\dotsc, a_n$ be the roots of an irreducible polynomial $f\in\Z[x]$. Let $s\geq 4\log n$. If $n$ is larger than some absolute constant, then there exists a prime $q\in (s,2s]$ such that $a_i/a_j$ is not a $q$-th root of unity for any $i\neq j$.
\end{lemma}

\begin{proof}[Proof of Theorem \ref{thm: konyagin fourier bound}]
Take $L=Cd\log p\log^{4}(d\log p)$. First, suppose that
\begin{equation}
\label{eq: large S_n}
    \sum_{n=n_0}^{n_0+L} (\widehat{S}_n)^2>\frac{p^2}{8\log(4L)} 
\end{equation}
for all $n_0$. Then we claim that
\begin{equation}
\label{eq: konyagin fourier bound}
    |\widehat{\nu_n}(\beta)|\leq \exp\left(\frac{n(1-\|\mu\|_2^2)}{Cd\log  p\log^5 (d\log p)}\right)
\end{equation}
for some absolute constant $C>0$. The result follows upon noticing that $1-\|\mu\|_2^2 \ge 1-\sum_{x} \mu(x)(1-\eta) = \eta$. 

To see that \eqref{eq: konyagin fourier bound} holds, note that
\begin{equation*}
    |\widehat{\nu_n}(\beta)|=\prod_{m=1}^{\frac{n}{L+1}} |\widehat{\mu_{m}}(\beta)|,
\end{equation*}
where $\mu_m$ is the distribution of the random variable 
\begin{equation*}
    \left(X_mD^{(k)}(\alpha_j^{m})\right)_{j\leq H, k\in\mathcal{K}_j}\in V.
\end{equation*}

First, note that $\cos(2\pi x)\leq 1-8\|x\|^2$ for all $x$, where $\|x\|$ denotes the distance to the closest integer. Then if $\varepsilon_1,\varepsilon_2$ are independently drawn from $\mu$,
\begin{equation*}
\begin{split}
    1-|\widehat{\mu_m}(\beta)|^2&=1-\E\left[e_p\left(\sum_j\sum_{k\in \mathcal{K}_j}\Tr\left(\beta_{j,k}(\varepsilon_1-\varepsilon_2)\alpha_j^{m-k}{m\choose k}\right)\right)\right]
    \\&=1- \sum_x \P[\varepsilon_1-\varepsilon_2=x]\cos\left(\frac{2\pi}{p}\sum_j\sum_{k\in \mathcal{K}_j}\Tr\left(\beta_{j,k}x\alpha_j^{m-k}{m\choose k}\right)\right)
    \\&\geq 8\sum_{x\neq 0}\P[\varepsilon_1-\varepsilon_2=x]\left\|\frac{1}{p}\sum_j\sum_{k\in \mathcal{K}_j}\Tr\left(\beta_{j,k}x\alpha_j^{m-k}{m\choose k}\right)\right\|^2
\end{split}
\end{equation*}
and so (after replacing the $\beta_{j,k}$ with $x\beta_{j,k}$),
\begin{equation*}
    |\widehat{\mu_m}(\beta)|\leq \exp\left(-4\sum_{x\neq 0}\P[\varepsilon_1-\varepsilon_2=x]\frac{\widehat{S}_m^2}{p^2}\right)\leq \exp\left(-4\frac{(1-\|\mu\|_2^2)\widehat{S}_m^2}{p^2}\right).
\end{equation*}
But \eqref{eq: large S_n} gives a uniform bound, and so by grouping the $|\widehat{\mu_m}(\beta)|$ into blocks of length $L$, we obtain
\begin{equation*}
    |\widehat{\mu_n}(\beta)|\leq\exp\left(\frac{n(1-\|\mu\|_2^2)}{L\log L}\right)\leq  \exp\left(\frac{n(1-\|\mu\|_2^2)}{Cd\log p\log^5(d\log p)}\right)
\end{equation*}
as required.

Otherwise, if there is an $n_0$ such that $\eqref{eq: large S_n}$ does not hold, by Proposition \ref{prop:konyagin}, for each $\alpha_j$, there exists a polynomial $f_j\in\Z[x]$ of degree at most $3d\log p$ with $f_j(\alpha_j)=0$ and 
\begin{equation*}
    M(f_j)\leq (d\log p)^{\frac{30d\log p}{L}}\leq e^{\frac{1}{Cd\log p\log^3 (d\log p)}}.
\end{equation*}
But by Dobrowolski's bound \cite{D79}, either
\begin{equation*}
    M(f_j)\geq e^{c\frac{\log^3 \log \deg f_j}{\log^3 \deg f_j}}\geq e^{c\frac{1}{\log^3 (d\log p)}}
\end{equation*}
for some absolute constant $c>0$, or $f_j$ is a product of cyclotomic polynomials. By taking $C$ large enough, we ensure the latter, and so $\alpha_j$ is the root of a cyclotomic polynomial of degree at most $3d\log p$. Since a cyclotomic polynomial of order $n$ has degree $\phi(n)\geq \frac{cn}{\log\log n}$ for some $c>0$, this implies $\alpha_j$ has order at most $Cd\log p\log(d\log p)$ for some constant $C>0$.
\end{proof}

\section{A total variation bound in terms of moments}
\label{sec: tv bound}
We begin by proving Proposition~\ref{prop: total variation bound}. 
\begin{proof}[Proof of Proposition \ref{prop: total variation bound}]
By Fourier inversion,
\begin{equation*}
\begin{split}
    &|\P[Z=a]-\P[Z'=a]|
    \\=&\left|\frac{1}{(2\pi)^N}\int _{[-\pi,\pi]^N}(\E[e^{i\theta\cdot Z}]-\E[e^{i\theta\cdot Z'})]e^{-i\theta\cdot a}d\theta\right|
    \\\leq& \frac{1}{(2\pi)^N}\int _{[-\pi,\pi]^N}\sum_{k=0}^{H-1} \frac{|\E(i\theta\cdot Z)^k-\E(i\theta\cdot Z')^k|}{k!}
    \\&\qquad\qquad+\E\left|\sum_{k=H+1}^\infty\frac{(i\theta\cdot Z)^k}{k!}\right|+\E\left|\sum_{k=H+1}^\infty\frac{(i\theta\cdot Z')^k}{k!}\right|d\theta
\end{split}  
\end{equation*}
The first term is bounded by $N^{H-1}e^{\pi}\varepsilon$ by expanding the powers and using \eqref{eq: moment error assumption}. The second and third terms are bounded by $\frac{C\pi ^H}{H!}$ by Taylor's theorem and \eqref{eq: moment tail assumption}.
\end{proof}

To apply Proposition \ref{prop: total variation bound}, we will need an estimate for the moments of $\sum_{i\leq N} N_i(\overline{f})$ under the uniform model to bound $C$.

\begin{lemma}
\label{lem: moment tail bound}
Let $\overline{f}$ be a uniformly random monic polynomial of degree $n$ in $\F_q[x]$. Then for all $H$ and $N$ with $H>\log(N+1)$,
\begin{equation*}
    \E\left(\left(\sum_{i\leq N} N_i(\overline{f})\right)^H\right)\leq \left(\frac{H}{\log H-\log \log(N+1)}\right)^H.
\end{equation*}

\end{lemma}
\begin{proof}
Note that it suffices to study the same problem for $\overline{f}$ a uniform monic random polynomial, since the bound is uniform in $n$. Let $X=\sum_{i\leq N} N_i(\overline{f})$ be the total number of distinct irreducible factors of a uniformly random monic polynomial up to degree $N$. The number of distinct irreducible factors of degree $i$ has the generating function (see \cite{KK93} for example)
\begin{equation*}
    D(y;z)=\sum_{g}y^{\deg(g)}\prod z_i^{N_i(g)}=\prod \left(1+\frac{z_iy^i}{1-y^i}\right)^{\pi(i)},
\end{equation*}
where $\pi(i)$ denotes the number of monic irreducible polynomials of degree $i$ in $\F_q[x]$. Then
\begin{equation*}
    \E[X^H]=H![t^H][y^n]D(y/q;e^t, \dotsc, 1,\dotsc),
\end{equation*}
where $z_i=e^t$ for all $i\leq N$, and $1$ otherwise, and $[z^n]F(z)$ denotes the coefficient of $z^n$ in the power series $F(z)$.

Given two formal power series $F$ and $G$, write $F\preceq G$ if all coefficients of $F$ are at most the corresponding coefficients in $G$. It can be seen that this is a partial order which respects the ring structure on formal power series, and respects composition when the coefficients are non-negative. Also, we have
\begin{equation*}
    \left(1+\frac{x}{n}\right)^n\preceq e^x.
\end{equation*}

Now as 
\begin{equation*}
    D(y/q;1,\dotsc)=\prod_i \left(\frac{1}{1-(y/q)^i}\right)^{\pi(i)}=\frac{1}{1-y},
\end{equation*}
we have
\begin{equation*}
\begin{split}
    D(y/q;e^t,\dotsc, 1,\dotsc)&=\frac{1}{1-y}\prod _{i=1}^N\left(1+(e^t-1)(y/q)^i\right)^{\pi(i)}
    \\&\preceq \frac{1}{1-y}\exp\left((e^t-1)\sum_{i\leq N}\frac{y^i}{q^i}\pi(i)\right)
    \\&\preceq \frac{1}{1-y}\exp\left((e^t-1)\sum_{i\leq N}\frac{y^i}{i}\right)
\end{split}
\end{equation*}
since $\pi(i)\leq q^i/i$. Then
\begin{equation*}
\begin{split}
    [t^H][y^N]\frac{1}{1-y}\exp\left((e^t-1)\sum_{i\leq N}\frac{y^i}{i}\right)&=\sum _{i=0}^N [y^i][t^H]\exp\left((e^t-1)\sum_{i\leq N}\frac{y^i}{i}\right)
    \\&\leq \sum _{i=0}^\infty [y^i][t^H]\exp\left((e^t-1)\sum_{i\leq N}\frac{y^i}{i}\right)
    \\&=[t^H]\exp\left({(e^t-1)\sum _{i\leq N}\frac{1}{i}}\right).
\end{split}
\end{equation*}
But this is just the moment generating function of a Poisson random variable of mean $\sum_{1\le i \le N}\frac{1}{i}\leq \log(N+1)$, and the result follows from the bound
\begin{equation*}
    \E[Z^H]\leq \left(\frac{H}{\log\left(\frac{H}{\E[Z]}+1\right)}\right)^H
\end{equation*}
for the moments of a Poisson random variable $Z$ (see \cite{A21}).
\end{proof}

\section{Moments of \texorpdfstring{$N_i(f)$}{Ni(f)} and \texorpdfstring{$N_i'(f)$}{Ni'(f)}}
\label{sec: moments}
Recall the general setup of Section \ref{sec: setup}. We begin with the following bounds on the difference in probabilities between the uniform and non-uniform model that the $\alpha_i$ are roots of $f$ of multiplicity $k_i$.

We define the following notion of high order elements and low order elements $\alpha\in \F_{p^e}$ which will be used repeatedly.
\begin{definition}[High and low order roots]
\label{def: order of roots}
Recall that we have parameters $H$ controlling the number of roots considered at a time, and $K$ controlling the number of derivatives we consider. We say that $\alpha\in\F_{p^e}$ has \emph{high order} if $\alpha$ has multiplicative order at least $m_e=CH(K+1)e\log p \log (H(K+1)e \log p)$ for some large constant $C$. Otherwise, we say that $\alpha$ has \emph{low order}. We say that an irreducible polynomial $g\in \F_p[x]$ is high or low order if all its roots are high or low order.
\end{definition}
The utility of this definition is that for any collection of at most $H$ high order roots, the second case of Theorem \ref{thm: konyagin fourier bound} cannot hold, and so we obtain strong bounds on the Fourier coefficients.

\begin{proposition}
\label{prop: prob discrep bound} We use the same notation as in previous sections. 
\begin{enumerate}
    \item Suppose that $d\leq n$. Then
\begin{equation*}
    \left|\P\left[(x-\alpha_i)^{k_i+1}|f(x) \text{ }\forall i\right]-\P\left[(x-\alpha_i)^{k_i+1}|\overline{f}(x) \text{ } \forall i\right]\right|\leq \exp\left(-\frac{cn}{dp^2}\right).
\end{equation*}
    \item If in addition, all $\alpha_i\in\F_{p^{e_i}}$ have high order, then
\begin{equation*}
    \left|\P\left[(x-\alpha_i)^{k_i+1}|f(x)\text{ }\forall i\right]-\P\left[(x-\alpha_i)^{k_i+1}|\overline{f}(x)\text{ }\forall i\right]\right|\leq \exp\left(-\frac{\eta n}{Cd\log p\log^5 (d\log p)}\right).
\end{equation*}
    
\end{enumerate}

\end{proposition}

\begin{proof}
We note that if $d\leq n$, then $\P\left[(x-\alpha_i)^{k_i}|\overline{f}(x)  \text{ } \forall i\right]=p^{-d}$. Now
\begin{equation*}
    \left|\P\left[(x-\alpha_i)^{k_i+1}|f(x)\text{ } \forall i\right]-p^{-d}\right|\leq \frac{1}{p^d}\sum _{\beta\neq 0}|\widehat{\nu_n}(\beta)|\leq \exp\left(-\frac{cn}{dp^2}\right)
\end{equation*}
by Proposition \ref{prop: weak fourier bound}. This proves the first part of the proposition.

Now suppose in addition that each $\alpha_i$ has high order. We then proceed as above, but use Theorem \ref{thm: konyagin fourier bound} instead. 
If all entries in $\beta$ are non-zero, this is immediate, since $H(K+1)\max e_i\geq d$.
If $\beta$ has entries which are $0$, we simply apply Theorem \ref{thm: konyagin fourier bound}, forgetting about those entries which are $0$. Note that the parameter $d$ will change, and the definition of high order for the $\alpha_i$ ensures that 
\begin{equation*}
    H(K+1)\max_{\beta_{i,k}\neq 0\text{ for some $k$}} e_i\geq d',
\end{equation*}
where $d'$ has the same definition as $d$, except we restrict to the $\alpha_i$ and $K_i$ for which $\beta_{i,k}\neq 0$ for all $k_i\in \mathcal{K}_i$. Thus, we always have at least one $\alpha_i$ with higher order than in the second case of Theorem \ref{thm: konyagin fourier bound}, which ensures
\begin{equation*}
    |\widehat{\nu}_n(\beta)|\leq \exp\left(-\frac{\eta n}{Cd\log p\log^5 (d\log p)}\right).
\end{equation*}
The second part of the proposition is now immediate.
\end{proof}

\subsection{Moments of \texorpdfstring{$N_i(f)$}{Ni(f)}}
We are now in a position to approximate the moments of $N_i(f)$, by expanding into events that can be controlled by Proposition \ref{prop: prob discrep bound}.

\begin{proposition}
\label{prop: moment discrep bound no mult}
Let $\sum h_i\leq H$ and assume that $NH\leq n$. Then we have
\begin{equation*}
    \left|\E \left[\prod_{i\leq N} N_i(f)^{h_i}\right]-\E \left[\prod_{i\leq N} N_i(\overline{f})^{h_i}\right]\right|\leq p^{HN}\exp\left(-\frac{cn}{NHp^2}\right).
\end{equation*}
If we define $\overline{N}_i(f)$ to be the number of high order factors of $f$ of degree $i$, then
\begin{equation*}
    \left|\E \left[\prod_{i\leq N} \overline{N}_i(f)^{h_i}\right]-\E \left[\prod_{i\leq N} \overline{N}_i(\overline{f})^{h_i}\right]\right|\leq p^{HN}\exp\left(-\frac{\eta n}{CNH\log p\log^5 (NH\log p)}\right).
\end{equation*}
\end{proposition}
\begin{proof}
We write $N_i(f)=\sum_{\alpha}I_\alpha(f)$, where $I_\alpha(f)$ is the indicator for the event that $\alpha$ is a root of $f$, and the sum is over a choice of representative root for the irreducible polynomials of degree $i$. Expanding the product, we have a sum of at most $p^{HN}$ terms, and the discrepancy for each term can be bounded by the first part of Proposition \ref{prop: prob discrep bound} (with $K=0$), where we note that $d\leq NH$. 

For the second part, we simply sum the above to a sum over $\alpha$ of high order, and use the second part of Proposition \ref{prop: prob discrep bound}.
\end{proof}

\subsection{Moments of \texorpdfstring{$N_i'(f)$}{Ni'(f)}}
We use a similar argument to control the moments of $N_i'(f)$. The only complication is our lack of control on the tail, for which we simply use the crude bound $N_i'(f)\leq n$.

\begin{proposition}
\label{prop: moment discrep bound with mult}
Let $\sum h_i\leq H$ and assume that $NH(K+1)\leq n$. Then we have
\begin{equation*}
\begin{split}
     & \left|\E \left[\prod_{i\leq N} N'_i(f)^{h_i}\right]-\E \left[\prod_{i\leq N} N'_i(\overline{f})^{h_i}\right]\right|\\
     & \leq (K+1)^Hp^{HN}\exp\left(-\frac{cn}{NHKp^2}\right)+n^H\left(\frac{2}{p^{K+1}}+p^N\exp\left(-\frac{cn}{NKp^2}\right)\right).
\end{split}  
\end{equation*}
If we define $\overline{N}'_i(f)$ to be the number of high order factors of $f$ of degree $i$, counted with multiplicity, then
\begin{equation*}
\begin{split}
    & \left|\E \left[\prod_{i\leq N} \overline{N}'_i(f)^{h_i}\right]-\E \left[\prod_{i\leq N} \overline{N}'_i(\overline{f})^{h_i}\right]\right| \\
    & \leq (K+1)^Hp^{HN}\exp\left(-\frac{\eta n}{CNHK\log p\log^5 (NHK\log p)}\right)
    \\&\qquad+n^H\left(\frac{2}{p^{K+1}}+p^N\exp\left(-\frac{\eta n}{CNK\log p\log^5 (NK\log p)}\right)\right).
\end{split}   
\end{equation*}
\end{proposition}
\begin{proof}
On the event that all roots have multiplicity at most $K+1$, we may write $N'(f)=\sum _{\alpha}\sum_{k\leq K} I_{\alpha,k}(f)$ where $I_{\alpha,k}(f)$ is the event that $f$ has $\alpha$ as a root of multiplicity at least $k+1$. Then conditional on this event, the proof of Proposition \ref{prop: moment discrep bound no mult} gives the same bounds, although the number of terms is now bounded by $(K+1)^Hp^{NH}$ and $d\leq NH(K+1)$

To obtain the desired result, we simply show that the probability of obtaining even a single root of degree at most $N$ of multiplicity greater than $K+1$ is very small. For the uniform model, if $\alpha\in \F_{p^e}$, then we have
\begin{equation*}
    \P\left[(x-\alpha)^{K+2}|\overline{f}(x)\right]=\frac{1}{p^{(K+2)e}}.
\end{equation*}
By Proposition \ref{prop: prob discrep bound}, 
\begin{equation*}
    \P\left[(x-\alpha)^{K+2}|f(x)\right]\leq \frac{1}{p^{(K+2)e}}+\exp\left(-\frac{cn}{eKp^2}\right).
\end{equation*}
By a union bound, the probability that $f$ has even one root of degree at most $N$ and multiplicity at least $K+2$ is at most
\begin{equation*}
    \sum_{e\leq N} \left(\frac{1}{p^{(K+2)e}}+\exp\left(-\frac{cn}{eKp^2}\right)\right)\pi(e)\leq \frac{2}{p^{K+1}}+p^N\exp\left(-\frac{cn}{NKp^2}\right),
\end{equation*}
where $\pi(e)\leq p^e$ denotes the number of irreducible polynomials of degree $i$.

Since $N_i'(f)\leq n$, this means
\begin{equation*}
\begin{split}
    &\left|\E \left[\prod_{i\leq N} N'_i(f)^{h_i}\right]-\E \left[\prod_{i\leq N} N'_i(\overline{f})^{h_i}\right]\right|
    \\\leq&(K+1)^Hp^{HN}\exp\left(-\frac{cn}{NHKp^2}\right)+n^H\left(\frac{2}{p^{K+1}}+p^N\exp\left(-\frac{cn}{NKp^2}\right)\right).
\end{split}
\end{equation*}
The second part of the proposition is completely analogous.
\end{proof}

\begin{remark}
\label{rmk: F_q for weak bound}
In fact, the first parts of Propositions \ref{prop: prob discrep bound}, \ref{prop: moment discrep bound no mult} and \ref{prop: moment discrep bound with mult} hold even for finite fields $\F_q$ for $q$ a prime power, since the key input is Proposition \ref{prop: weak fourier bound} which does not require working over $\F_p$.
\end{remark}

\section{Proof of Main Results}
\label{sec: main results}
In this section, we state and prove the main theorems, and then prove some corollaries which were stated in the introduction.

\subsection{Statements of main results}
Recall that $\mu$ is a probability distribution on $\F_q$, where $q=p^e$, and $\eta=1-\max_{V\subseteq \F_q}\mu(V)$, where the maximum is taken over all proper $\F_p$ affine subspaces of $\F_q$. When $q=p$ is a prime, $\eta=1-\max_{x\in\F_p}\mu(x)$. All constants in what follows can depend on $\eta$.

We let $f(x)=\sum_{i=0}^n \varepsilon_i x^i$, where $\varepsilon_i$ are drawn independently from $\mu$. We let $\overline{f}$ denote a uniformly random monic polynomial of degree $n$. For a polynomial $g$, let $N_i(g)$ denote the number of distinct irreducible factors of degree $i$, and let $N'_i(g)$ denote the number of irreducible factors of degree $i$, counted with multiplicity, where we do not count the factor $x$ in $N_1(g)$ and $N'_1(g)$.

The first two theorems give quantitative bounds between $N_i(f)$ and $N_i(\overline{f})$ for the regimes where $p$ is small and where $p$ is large.
\begin{theorem}
\label{thm:main small p no mult}
Suppose that $\eta>0$ and $n\geq C N^4(\log n)^2p^2\log q$ for some large constant $C$ depending only on $\eta$. Then 
\begin{equation*}
    d_{TV}((N_i(f))_{i\leq N},(N_i(\overline{f}))_{i\leq N})=O\left(\exp\left(-cn^{\frac{1}{4}}p^{-\frac{1}{2}}\log^{-\frac{1}{4}} q\right)\right)
\end{equation*}
for some constant $c>0$ depending only on $\eta$.
\end{theorem}
\begin{theorem}
\label{thm: main large p no mult}
Suppose that $\mu$ is a distribution on $\F_p$ and $\eta>0$. Then
\begin{equation*}
    d_{TV}((N_i(f))_{i\leq N},(N_i(\overline{f}))_{i\leq N})=O\left(\left(p^{-1}+n^{-\frac{1}{2}}\right)N^2\log^4 n\log^2 p\right),
\end{equation*}
where the implicit constant depends only on $\eta$.
\end{theorem}

The next two theorems give quantitative bounds between $N_i'(f)$ and $N_i'(\overline{f})$ in the regimes where $p$ is small and where $p$ is large.
\begin{theorem}
\label{thm: main small p with mult}
Suppose that $\eta>0$ and $n\geq C N^5  (\log n)^{4} p^2\log q$ for some large constant $C$ depending only on $\eta$. Then
\begin{equation*}
    d_{TV}((N'_i(f))_{i\leq N},(N'_i(\overline{f}))_{i\leq N})=O\left(\exp\left(-cn^{\frac{1}{5}}p^{-\frac{2}{5}}\log^{-\frac{1}{5}}q\right)\right)
\end{equation*}
for some constant $c>0$ depending only on $\eta$.
\end{theorem}
\begin{theorem}
\label{thm: main large p with mult}
Suppose that $\mu$ is a distribution on $\F_p$ and $\eta>0$. Then
\begin{equation*}
    d_{TV}((N_i(f))_{i\leq N},(N_i(\overline{f}))_{i\leq N})=O\left(\left(p^{-1}+n^{-\frac{1}{2}}\right)N^4\log^8 n\log^2 p\right),
\end{equation*}
where the implicit constant depends only on $\eta$.
\end{theorem}

\begin{remark}
The polynomial error in Theorems \ref{thm: main large p no mult} and \ref{thm: main large p with mult} should be necessary, at least for large enough $p$. To see this, note that the probability that $1$ is a root should be of order $n^{-\frac{1}{2}}$ by comparing to a normal using a local central limit theorem as long as $p>n^2$. Since the number of low order roots in $\F_p$ is roughly $\log p$ if we only consider roots in $\F_p$, we would expect $\E(N_1(f))\geq 1+O(n^{-\frac{1}{2}})$ with a polynomial rather than exponential error. Some numerical simulations given in Section \ref{sec: numerics} also support this.
\end{remark}

\subsection{Proof of main results}
\begin{proof}[Proof of Theorem \ref{thm:main small p no mult}]
We apply Proposition \ref{prop: total variation bound} to the random variables $N_i(f)$ for $i\leq N$, taking $H=N\log n$. We first note that by the first part of Proposition \ref{prop: moment discrep bound no mult} (and Remark \ref{rmk: F_q for weak bound}), we may take $\varepsilon=q^{NH}\exp\left(-\frac{cn}{HNp^2}\right)$. Then using Lemma \ref{lem: moment tail bound} to bound $C$, and summing over the support of the $N_i(f)$, which has size at most $(n+1)^N$ since $N_i(f)\leq n$, we obtain a bound for $d_{TV}((N_i(f))_{i\leq N},(N_i(\overline{f}))_{i\leq N})$ up to a constant of the form
\begin{equation*}
\begin{split}
    &\exp\left(N\log (n+1)+H\log N +HN\log q-\frac{cn}{HNp^2}\right)
    \\&\qquad+\exp\left(N\log(n+1) +CH-H\log\log H\right).
\end{split}
\end{equation*}
Since $H=N\log n$, the second term is $O(e^{-cN\log n})$, and since $n\gg N^4(\log n)^2p^2\log p$, the first term is $O\left(\exp\left(-\frac{cn}{HNp^2}\right)\right)=O(e^{-cN\log n})$. Finally, there's no harm in assuming that $N$ is as large as possible, since total variation distance cannot increase under projection, and so taking $N=cn^{\frac{1}{4}}\log^{-1}np^{-\frac{1}{2}}\log^{-\frac{1}{4}}q$ gives the desired bound.
\end{proof}

The proof of Theorem \ref{thm: main large p no mult} relies on the fact that when $p$ is large, with high probability, there are no low order roots. The following lemma allows us to use this to obtain total variation bounds by conditioning on this event.

\begin{lemma}
\label{lem: TV bound}
Let $X,X':\Omega_X\to S$ and $Y,Y':\Omega_Y\to S$ be random variables into some set $S$. Then
\begin{equation*}
    d_{TV}(X,Y)\leq d_{TV}(X',Y')+2\P(X\neq X')+2\P(Y\neq Y').
\end{equation*}
\end{lemma}
\begin{proof}
We have for any event $A\subseteq S$,
\begin{equation*}
\begin{split}
    &|\P(X\in A)-\P(Y\in A)|
    \\\leq &|\P(X\in A, X=X')-\P(Y\in A, Y=Y')|+\P(X\neq X')+\P(Y\neq Y')
    \\=&|\P(X'\in A, X=X')-\P(Y'\in A, Y=Y')|+\P(X\neq X')+\P(Y\neq Y')
    \\\leq &|\P(X'\in A)-\P(Y'\in A)|+2\P(X\neq X')+2\P(Y\neq Y').
\end{split}
\end{equation*}
The desired inequality follows by taking the supremum over $A$ on both sides.
\end{proof}

The next lemma shows that when $p$ is large, with high probability, there are no low order roots.

\begin{lemma}
\label{lem: bound on prob of low order roots}
Let $f(x)=\sum_{i=0}^n \varepsilon_i x^i$, and suppose that $N=o(n)$. The probability that $f$ has a low order root (in the sense of Definition \ref{def: order of roots}) of degree at most $N$ is bounded by
\begin{equation*}
    C \left(p^{-1}+\eta^{-\frac{1}{2}}n^{-\frac{1}{2}}\right)H^2K^2\log^2 p\log^2(HK \log p).
\end{equation*}
\end{lemma}
\begin{proof}
We see that by Proposition \ref{prop:halasz-bound}, if $\alpha\in \F_{p^i}$ is low order, then
\begin{equation*}
    \P(f(\alpha)=0)\leq \left(p^{-1}+C\eta^{-\frac{1}{2}}\left(\frac{n}{i}\right)^{-1/2}\right)^i.
\end{equation*}

Since the number of low order roots in $\F_{p^i}$ is bounded by $$m_i^2\leq CH^2K^2i^2\log^2 p\log^2(HKi \log p),$$ a union bound gives that the probability that $f$ has a low order root is bounded by
\begin{equation*}
    \sum _{i=1}^N m_i^2\left(p^{-1}+C\eta^{-\frac{1}{2}}\left(\frac{n}{i}\right)^{-1/2}\right)^i.
\end{equation*}

Note that 
\[
m_i^2/m_1^2 = i^2\log^2(HKi\log p)/\log^2(HK\log p) \le C i^{2.5},
\] while 
\[
\left(p^{-1}+C\eta^{-\frac{1}{2}}\left(\frac{n}{i}\right)^{-\frac{1}{2}}\right)^i \le \left(p^{-1}+C\eta^{-\frac{1}{2}}n^{-\frac{1}{2}}\right) \cdot i^{1/2} 3^{-i+1}
\]
for large enough $n$, since $i=o(n)$. Thus 
\begin{equation*}
\begin{split}
    &\sum_{i} m_i^2 \left(p^{-1}+C\eta^{-\frac{1}{2}}\left(\frac{n}{i}\right)^{-\frac{1}{2}}\right)^i \\
    &\le C \sum_{i} i^33^{-i+1}m_1^2\left(p^{-1}+C\eta^{-\frac{1}{2}}n^{-\frac{1}{2}}\right) \\
    &\le C \left(p^{-1}+\eta^{-\frac{1}{2}}n^{-\frac{1}{2}}\right)H^2K^2\log^2 p\log^2(HK \log p).
\end{split}
\end{equation*}
\end{proof}

With these lemmas, we can now prove Theorem \ref{thm: main large p no mult} in exactly the same way as we did \ref{thm:main small p no mult}, but counting only the high order roots.

\begin{proof}[Proof of Theorem \ref{thm: main large p no mult}]
We proceed as in the proof of Theorem \ref{thm:main small p no mult}, taking $H=N\log n$, except we consider the random variables $\overline{N}_i(f)$ and $\overline{N}_i(\overline{f})$, the number high order roots of degree $i$. Then after using the second part of Proposition \ref{prop: moment discrep bound no mult}, we obtain a bound for $d_{TV}((\overline{N}_i(f))_{i\leq N},(\overline{N}_i(\overline{f}))_{i\leq N})$ up to a constant by
\begin{equation*}
    \begin{split}
      &\exp\left(N\log(n+1) +H\log N+HN\log p-\frac{c(1-\|\mu\|_2^2)n}{NH\log p \log^5(NH\log p)}\right)
        \\&\qquad+\exp\left(N\log(n+1) +CH-H\log\log H\right).
    \end{split}
\end{equation*}
We may assume that $n\geq C N^4\log^7 n\log^2 p$ and $p=O(\exp(n^{\frac{1}{4}}))$ as otherwise the claimed upper bound is vacuous, and so the first term is 
\begin{equation*}
    O\left(\exp\left(-\frac{c(1-\|\mu\|_2^2)n}{N^2\log n\log p \log^5(N^2\log n\log p)}\right)\right)=O(e^{-cN\log n}),
\end{equation*}
and as $H=N\log n$, the second term is $O(e^{-cN\log n})$ as before. Thus,
\begin{equation*}
    d_{TV}((\overline{N}_i(f))_{i\leq N},(\overline{N}_i(\overline{f}))_{i\leq N})=O(e^{-cN\log n}).
\end{equation*}
Again, there's no harm in taking $N$ larger as long as the assumed inequality holds, and choosing $N=n^{\frac{1}{4}}\log^{-\frac{7}{4}}n\log^{-\frac{1}{2}}p$ gives an error of $\exp\left(-cn^{\frac{1}{4}}\log^{-\frac{3}{4}}n\log^{-\frac{1}{2}}p\right)$. Finally, we note that $p=O(\exp(n^{\frac{1}{4}}))$ and then this error is dominated by the claimed upper bound.

Finally, as $N_i(f)=\overline{N}_i(f)$ if we have no low order roots (and similarly for $\overline{f}$), by Lemmas \ref{lem: TV bound} and \ref{lem: bound on prob of low order roots} (applied to both $f$ and $\overline{f}$) and the bound just obtained, the result follows, again using that we may assume $p\leq e^n$ to simplify the upper bound.
\end{proof}

\begin{proof}[Proof of Theorem \ref{thm: main small p with mult}]
Take $H=N\log n$ and $K=N\log^2 n$. Then the proof is analogous to the proof of Theorem \ref{thm:main small p no mult}. We use Proposition \ref{prop: total variation bound} together with the moment bounds of Proposition \ref{prop: moment discrep bound with mult}, where the upper bound simplifies to an $O(e^{-cN\log n})$ bound from the assumption on $n$ and the choice of $H$ and $K$, and again we can take $N$ as large as possible without issue.
\end{proof}

\begin{proof}[Proof of Theorem \ref{thm: main large p with mult}]
Take $H=N\log n$ and $K=N\log^2n$. The proof is then analogous to that of Theorem \ref{thm: main large p no mult}. We first assume $n\geq CN^5\log^9 n\log^2p$ and $p=O(\exp(n^{\frac{1}{4}}))$ as otherwise the claimed upper bound is vacuous. We use Proposition \ref{prop: total variation bound}, together with Proposition \ref{prop: moment discrep bound with mult}, where the upper bound simplifies to an $O(e^{-cN\log n})$ bound by the assumption on $n$ and the choice of $H$ and $K$. We may then choose $N$ to be large as long as the above inequality holds, so we take $N=n^{\frac{1}{4}}$, at which point this error is dominated by the claimed upper bound.

Finally, we proceed as in the proof of Theorem \ref{thm: main large p no mult}, using Lemmas \ref{lem: bound on prob of low order roots} and \ref{lem: TV bound} to restrict to the event that there are no low order roots, and again use $p=O(\exp(n^{\frac{1}{4}}))$ to obtain the desired bound.
\end{proof}

\subsection{Unconditional bounds}
We showed two different bounds for the total variation distance between $N_i(f)$ and $N_i'(\overline{f})$ that work well in different regimes. As a corollary, we can derive total variation bounds for a very large range of $p$. Here, we prove Corollaries \ref{cor: no mult}, \ref{cor: with mult} and \ref{cor: with q}. We present results maximizing $N$, the maximum degree of the irreducible factors we can control. One could take $p$ as large as $e^{n^{\frac{1}{4}-\delta}}$ by taking smaller $N$.

\begin{proof}[Proof of Corollary \ref{cor: no mult}]
First, suppose that $p\leq n^{\frac{1}{4}}$. Then taking $N=n^{\frac{1}{8}-\delta}$, for large enough $n$ the condition of Theorem \ref{thm:main small p no mult} is satisfied and the result follows.

If $n^{\frac{1}{4}}\leq p\leq e^{n^{\frac{1}{8}}}$, then taking $N=n^{\frac{1}{8}-\delta}$, the condition of Theorem \ref{thm: main large p no mult} holds, and the error bound obtained is bounded by $O(n^{-\delta})$.
\end{proof}

\begin{proof}[Proof of Corollary \ref{cor: with mult}]
First, suppose that $p\leq n^{\frac{4}{13}}$. Then taking $N=n^{\frac{1}{13}-\delta}$, for large enough $n$ the condition of Theorem \ref{thm: main small p with mult} is satisfied and the result follows.

Otherwise, if $n^{\frac{4}{13}}\leq p\leq e^{n^{\frac{1}{13}}}$, then taking $N=n^{\frac{1}{13}-\delta}$, the conditions of Theorem \ref{thm: main large p with mult} hold, and the error bound obtained is bounded by $O(n^{-3\delta})$.
\end{proof}

We can also obtain stronger results if $p$ is fixed, and in fact this even works over $\F_q$.

\begin{proof}[Proof of Corollary \ref{cor: with q}]
This result immediately follows from Theorems \ref{thm:main small p no mult} and \ref{thm: main small p with mult}.
\end{proof}

\subsection{Number of irreducible factors of fixed degree}

\begin{theorem}\label{thm:1-degree}
The following statements hold for positive absolute constants $c,C>0$. 
\begin{enumerate}
    \item Suppose $(\log p)^2 \log^{10}(n\log p)i^2 \le n$, and $i>10(\log n)/(\log p)$. Then $$d_{TV}(N_i(f),N_i(\overline{f})) \le \exp\left(- \frac{\eta}{C} \left(\frac{n}{i^2 (\log p)^2 (\log n)^5 }\right)^{1/2}\right) + \left(\frac{1}{p}+C\eta^{-1/2}\lfloor n/i\rfloor^{-1/2}\right)^{i/2}.$$
    \item Suppose $1\le i \le 10 (\log n)/(\log p)$ and $p>(\log n)^{100}$. Then 
    \begin{align*}
    d_{TV}(N_i(f),N_i(\overline{f})) 
    &\le \exp(-c\log n\log \log n) + \left(\frac{1}{p}+C\eta^{-1/2}\lfloor n/i\rfloor^{-1/2}\right)^{i/2}.
    \end{align*}
    \item Suppose $n \ge Ci^2 p^2 \log p (\log n)^2$. Then $$d_{TV}(N_i(f),N_i(\overline{f})) \le \exp\left(-\frac{\eta}{C} \cdot \frac{ n^{1/2}}{pi \sqrt{\log p}}\right).$$
\end{enumerate}

\end{theorem}
\begin{proof}
Throughout the argument we denote by $C$ absolute constants independent of all other parameters. 

We have from Proposition \ref{prop: moment discrep bound no mult} that for $h\le H$,
\[
\left|\E\left[\overline{N}_i(f)^{h}\right]-\E\left[\overline{N}_i(\overline{f})^{h}\right]\right| \le p^{Hi}\exp\left(-\frac{\eta n}{CHi \log p \log^{5}(H i \log p)}\right).
\]
Thus, by using Proposition~\ref{prop: total variation bound}, we have  
\begin{align*}
&d_{TV}(\overline{N}_i(f),\overline{N}_i(\overline{f})) \\
&\le C n p^{Hi}\exp\left(-\frac{\eta n}{CHi \log p \log^{5}(H i \log p)}\right) + C n \left(\frac{H}{\log H}\right)^H \frac{\pi^H}{H!}.
\end{align*}

Recall $m_i = CHi\log p \log(Hi\log p)$. The probability that $N_i(f) - \overline{N}_i(f) \ne 0$ is at most $$m_i^2 \left(\frac{1}{p}+C\eta^{-1/2}\lfloor n/i\rfloor^{-1/2}\right)^i,$$
and the same bound holds for $\overline{f}$. 
Then we have
\begin{align*}
&d_{TV}(\overline{N}_i(f),\overline{N}_i(\overline{f})) \\
&\le C n p^{Hi}\exp\left(-\frac{\eta n}{CHi \log p \log^{5}(H i \log p)}\right) + n(\log H)^{-H/2} + m_i^2 \left(\frac{1}{p}+C\eta^{-1/2}\lfloor n/i\rfloor^{-1/2}\right)^i.
\end{align*}
By choosing $H=c\left(\frac{n}{i^2(\log p)^2 \log^5(n\log p)}\right)^{1/2}$, we obtain 
\begin{align*}
&d_{TV}(\overline{N}_i(f),\overline{N}_i(\overline{f})) \\
&\le C n \exp\left(- \frac{\eta}{C} \left(\frac{n}{\log^{5}(n\log p)}\right)^{1/2}\right) + n(\log H)^{-H/2} + \frac{Cn}{(\log n)^3} \left(\frac{1}{p}+C\eta^{-1/2}\lfloor n/i\rfloor^{-1/2}\right)^i\\
&\le C \exp\left(- \frac{\eta}{C} \left(\frac{n}{i^2 (\log p)^2 (\log n)^5 }\right)^{1/2}\right) + \frac{Cn}{(\log n)^3} \left(\frac{1}{p}+C\eta^{-1/2}\lfloor n/i\rfloor^{-1/2}\right)^i.
\end{align*}
On the other hand, we can choose $H=i \log n$ and obtain 
\begin{align*}
&d_{TV}(\overline{N}_i(f),\overline{N}_i(\overline{f})) \\
&\le C n \exp\left(-\frac{\eta}{C} \cdot \frac{n}{i^2(\log n)(\log p)\log^{5}(n\log p)}\right) + \exp(-ci\log n\log \log n) \\
&\qquad \qquad + (i^2 \log n\log p)^3 \left(\frac{1}{p}+C\eta^{-1/2}\lfloor n/i\rfloor^{-1/2}\right)^i\\
&\le C n \exp\left(-\frac{\eta}{C} \cdot \frac{n}{i^2(\log n)^6(\log p)}\right) + n^{-ci} + (i \log n\log p)^6 \left(\frac{1}{p}+C\eta^{-1/2}\lfloor n/i\rfloor^{-1/2}\right)^i.
\end{align*}
If $i > 10\log_p(n)$, the first bound yields
\begin{align*}
&d_{TV}(\overline{N}_i(f),\overline{N}_i(\overline{f})) \\
&\le \exp\left(- \frac{\eta}{C} \left(\frac{n}{i^2 (\log p)^2 (\log n)^5 }\right)^{1/2}\right) + \left(\frac{1}{p}+C\eta^{-1/2}\lfloor n/i\rfloor^{-1/2}\right)^{i/2}.
\end{align*}
Otherwise, if $i\le 10\log_p(n)$ and $p>(\log n)^{100}$, the second bound yields
\begin{align*}
&d_{TV}(\overline{N}_i(f),\overline{N}_i(\overline{f})) \\
&\le \exp\left(-\frac{\eta}{C} \cdot \frac{n}{i(\log n)^6(\log p)}\right) + n^{-ci} + \left(\frac{1}{p}+C\eta^{-1/2}\lfloor n/i\rfloor^{-1/2}\right)^{i/2}.
\end{align*}

Furthermore, we always have
\[
\left|\E\left[{N}_i(f)^{h}\right]-\E\left[\overline{N}_i({f})^{h}\right]\right| \le p^{hi}\exp\left(-\frac{\eta n}{Chip^2}\right),
\]
so
\begin{align*}
&d_{TV}({N}_i(f),{N}_i(\overline{f})) \\
&\le C n p^{Hi}\exp\left(-\frac{\eta n}{CHi p^2}\right) + C n \left(\frac{H}{\log H}\right)^H \frac{\pi^H}{H!}.
\end{align*}
Choose $H=\left(\frac{n}{Ci^2p^2\log p}\right)^{1/2}$. Assuming that $H \ge \log n$, then
\[
d_{TV}({N}_i(f),{N}_i(\overline{f})) \le \exp\left(-\frac{\eta}{C} \left(\frac{n^{1/2}}{p/\sqrt{\log p}}\right)\right) + \exp(-H/2),
\]
from which we obtain 
\[
d_{TV}(N_i(f),N_i(\overline{f})) \le \exp\left(-\frac{\eta}{C} \cdot \frac{ n^{1/2}}{pi \sqrt{\log p}}\right).
\]
\end{proof}

The proof of Corollary~\ref{cor: min} follows from the following result.
\begin{corollary}
For $\epsilon > 0$, we have the following for $c$ a positive absolute constant independent of $\epsilon$ and $\eta$. Assume that $p<\exp(c n^{1/2-\epsilon}/i)$. Then 
\begin{enumerate}
    \item For $(\log n)^{2} \le i \le n^{1/2-2\epsilon}$,
\[
d_{TV}(N_i(f),N_i(\overline{f})) = O_{\eta,\epsilon}(\exp(-c\eta n^{1/2}/(i^2 (\log p)^2 (\log n)^5)^{1/2})+\exp(-ci)).
\]
\item  
For $i<(\log n)^{2}$,
\[
d_{TV}(N_i(f),N_i(\overline{f})) = O_{\eta,\epsilon}(\exp(-c\eta n^{1/2}/(i^2 (\log p)^2 (\log n)^5)^{1/2})+\exp(-n^{c})+n^{-ci}).
\]
\end{enumerate}
\end{corollary}
\begin{proof}
If $(\log n)^2 \le i \le n^{1/2-2\epsilon}$, then we have by the first claim in Theorem \ref{thm:1-degree} that 
\[
d_{TV}(N_i(f),N_i(\overline{f})) = O_{\eta,\epsilon}(\exp(-c\eta n^{1/2}/(i^2 (\log p)^2 (\log n)^5)^{1/2})+\exp(-ci)).
\]
If $i<(\log n)^{2}$, $p<n^{1/8}$, we can use the third claim in Theorem \ref{thm:1-degree} to obtain 
\[
d_{TV}(N_i(f),N_i(\overline{f})) = \exp(-n^{c}).
\]
If $i<(\log n)^{2}$, $p>n^{1/8}$ and $p<\exp(cn^{1/2-\epsilon}/i)$, we can use the first two claims in Theorem \ref{thm:1-degree} to obtain 
\[
d_{TV}(N_i(f),N_i(\overline{f})) = O_{\eta,\epsilon}(\exp(-c\eta n^{1/2}/(i^2 (\log p)^2 (\log n)^5)^{1/2}) + \exp(-n^{1/2-o(1)}) + n^{-ci}).
\]
\end{proof}

With an identical proof, we can also obtain similar conclusions for other linear statistics of the number of roots of each degree, such as the total number of roots of degree bounded by $N\le n^{1/2-o(1)}$.
\begin{corollary}
Let $T(f) = \sum_{i\le N}N_i(f)$. Then we have for $N\le n^{1/2-\epsilon}$, we have that
\[
d_{TV}(T(f),T(\overline{f})) = O_{\eta,\epsilon}(\exp(-n^{c\epsilon})+\exp(-cN)).
\]
\end{corollary}

\section{Numerical simulations}
\label{sec: numerics}
In this section, we provide some numerical simulations which both support the main results and also suggest some directions for further investigation. 

\subsection{Low degree factors}
We begin with some simulations supporting our results on the low degree factors.

\begin{figure}
    \centering
    \begin{subfigure}[]{0.45\textwidth}
    \centering
    \includegraphics[scale=0.5]{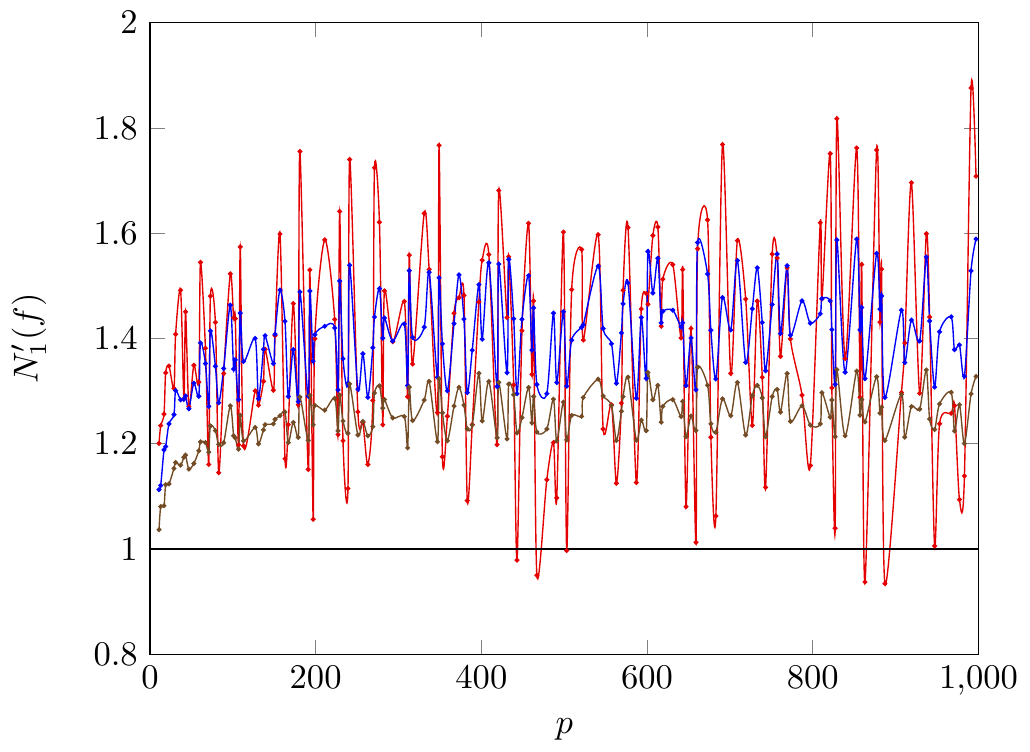}
    \subcaption{$i=1$, small $p$}
    \end{subfigure}
    \begin{subfigure}[]{0.45\textwidth}
    \centering
    \includegraphics[scale=0.5]{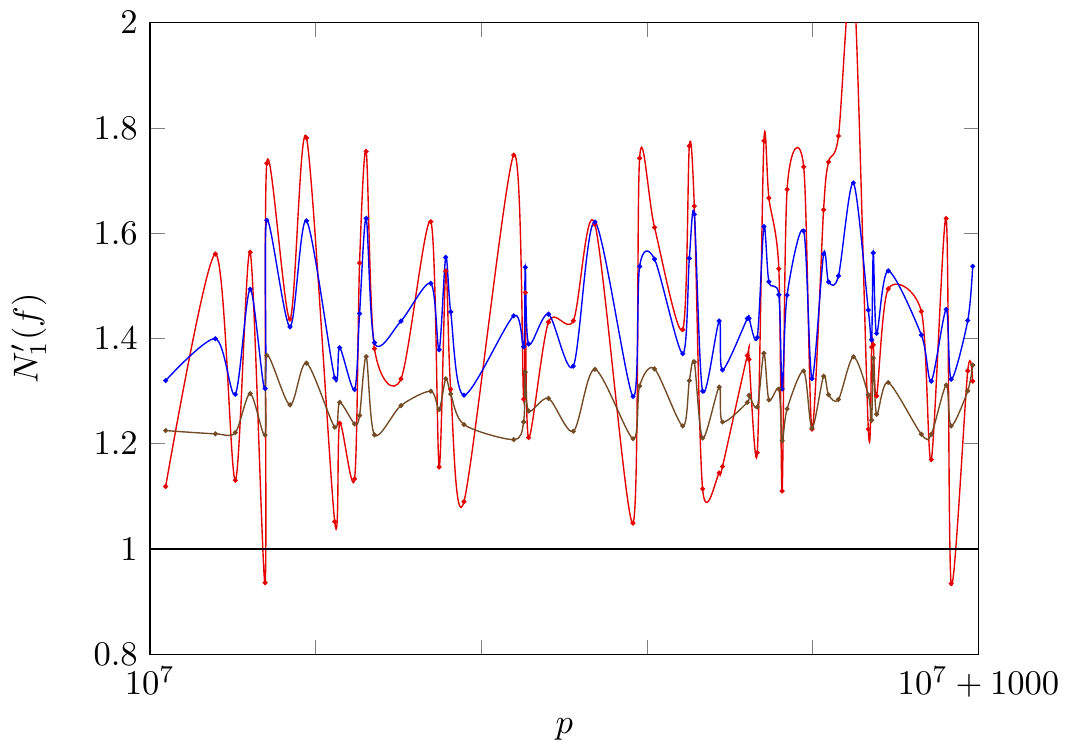}
    \subcaption{$i=1$, large $p$}
    \end{subfigure}
    \\
    \begin{subfigure}[]{0.45\textwidth}
    \centering
    \includegraphics[scale=0.5]{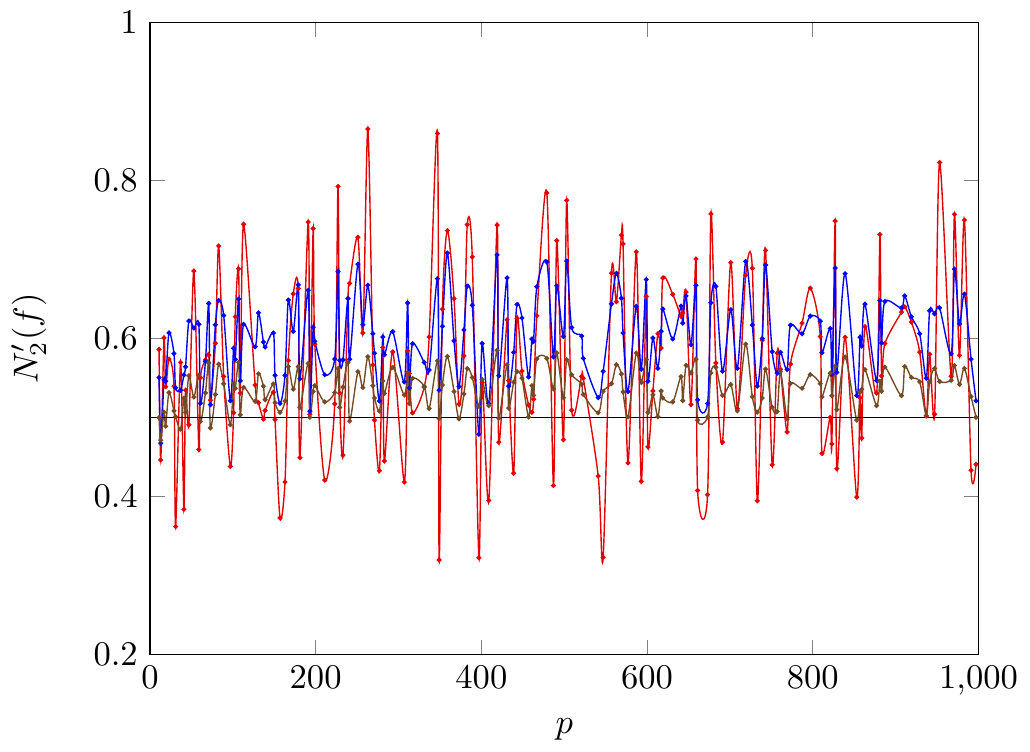}
    \subcaption{$i=2$, small $p$}
    \end{subfigure}
    \begin{subfigure}[]{0.45\textwidth}
    \centering
    \includegraphics[scale=0.5]{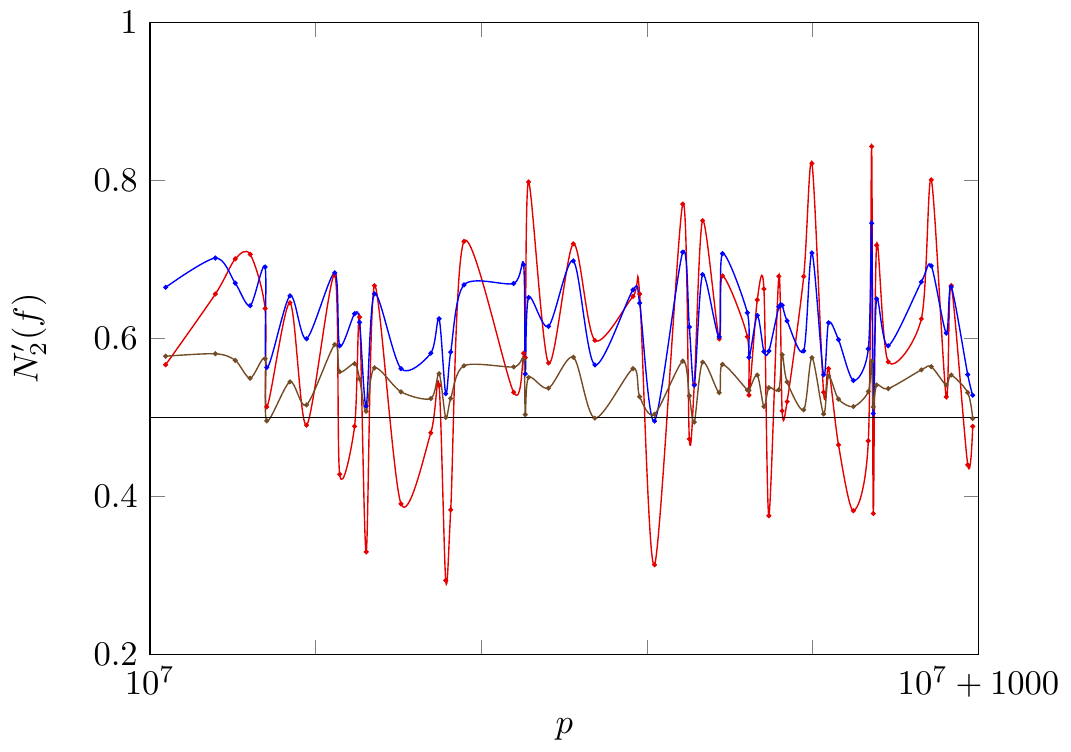}
    \subcaption{$i=2$, large $p$}
    \end{subfigure}
    \\
    \begin{subfigure}[]{0.45\textwidth}
    \centering
    \includegraphics[scale=0.5]{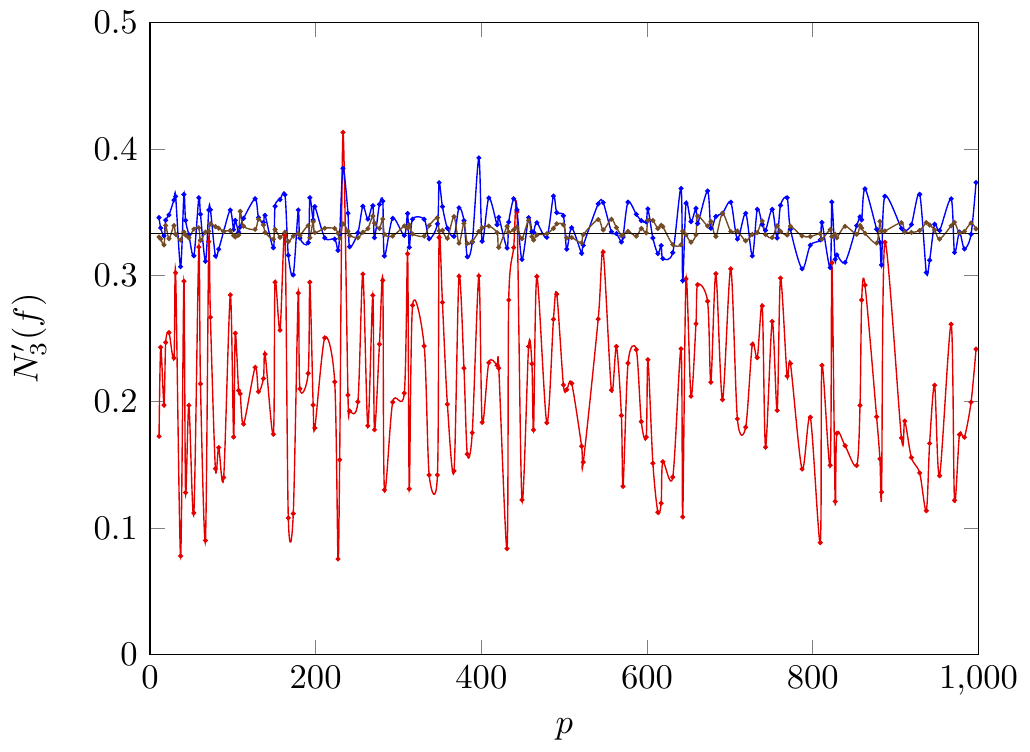}
    \subcaption{$i=3$, small $p$}
    \end{subfigure}
    \begin{subfigure}[]{0.45\textwidth}
    \centering
    \includegraphics[scale=0.5]{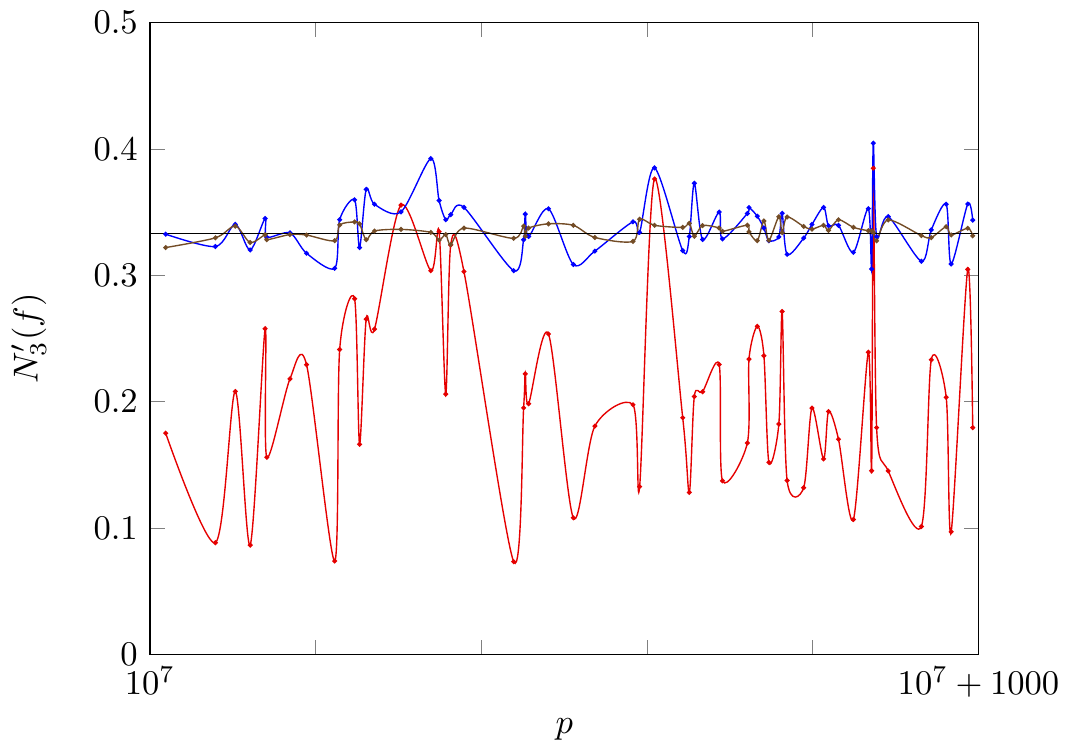}
    \subcaption{$i=3$, large $p$}
    \end{subfigure}
    \caption{Empirical means of $N'_i(f)$ with 10,000 trials when $\mu$ is the uniform measure on $\{-1,0,1\}$ on $\F_p$ for primes $p$ between $10$ and $1,000$ as well as between $10^7$ and $10^7+1000$. Here $f$ is a random polynomial of degree $n=5,10,20$, with $n=5$ in red, $n=10$ in blue, and $n=20$ in brown. The black line indicates $\frac{1}{i}$, which is a good approximation for the expected value under the uniform model for large $p$.}
    \label{fig: p dep}
\end{figure}

Figure \ref{fig: p dep} shows the empirical means of $N'_i(f)$ computed with 10,000 trials, for $i=1,2,3$ and primes $10\leq p\leq 1000$ as well as $10^7\leq p\leq 10^7+1000$. We show data for $f$ a random polynomial of degree $5$, $10$ or $20$, and with coefficients drawn uniformly from $\{-1,0,1\}$. 

This data shows that even when the degree $n$ is relatively small, the expectations are quite close to the true values, except for $i=1$, and that the error does not seem to deteriorate with $p$. Admittedly, our results indicate that any such deterioration should occur when $p\gg e^n$, and so it is possible that the behaviour changes for extremely large $p$. 

The fact that the error is larger for small $i$ also makes sense because in the regime $p\gg n$, the error should be dominated by the error coming from the low order roots, which is of order $n^{-i/2}$.

\begin{figure}
    \centering
    \begin{subfigure}[]{0.3\textwidth}
    \centering
    \includegraphics[scale=0.5]{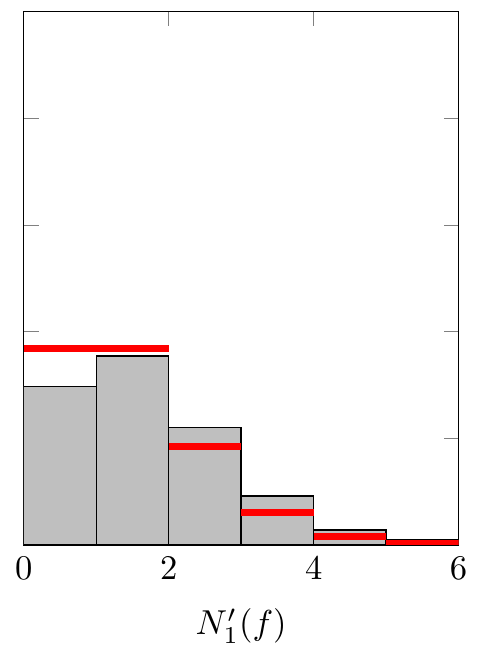}
    \subcaption{$n=20$, $i=1$}
    \end{subfigure}
    \begin{subfigure}[]{0.3\textwidth}
    \centering
    \includegraphics[scale=0.5]{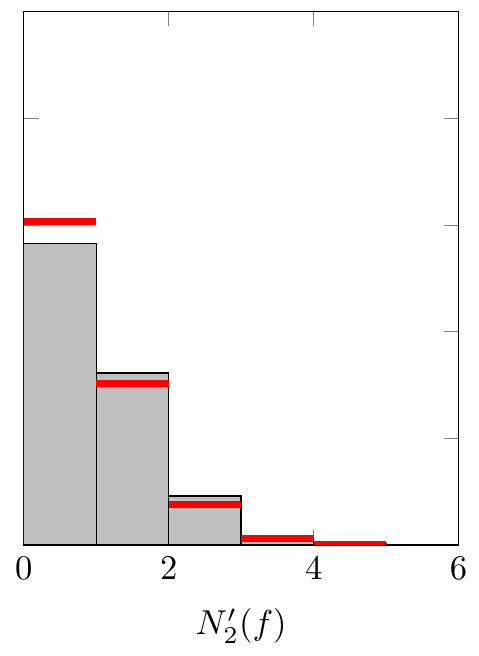}
    \subcaption{$n=20$, $i=2$}
    \end{subfigure}
    \begin{subfigure}[]{0.3\textwidth}
    \centering
    \includegraphics[scale=0.5]{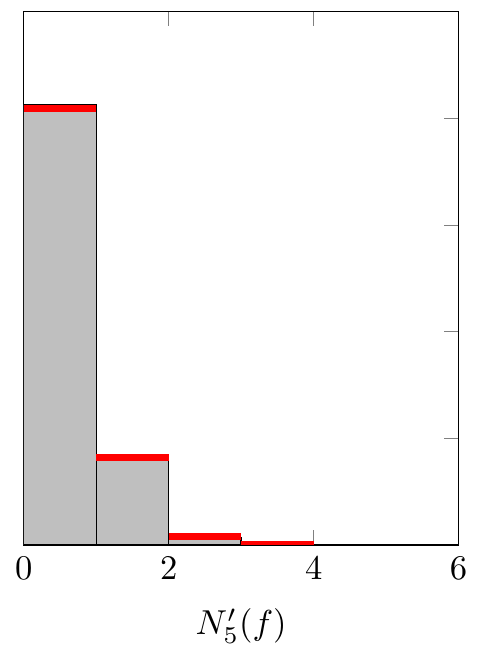}
    \subcaption{$n=20$, $i=5$}
    \end{subfigure}
    \\
    \begin{subfigure}[]{0.3\textwidth}
    \centering
    \includegraphics[scale=0.5]{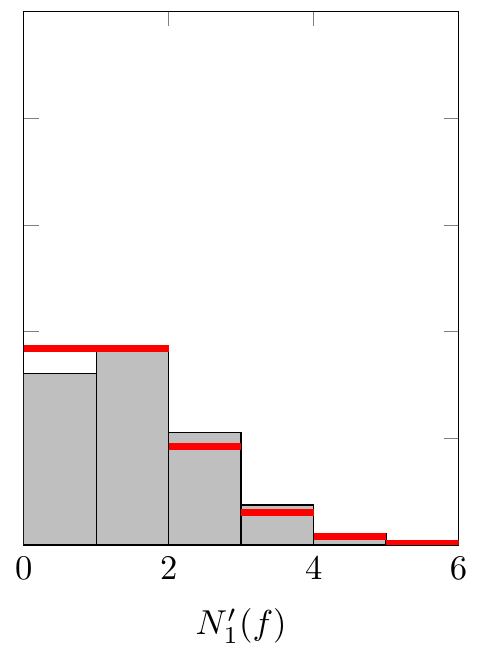}
    \subcaption{$n=50$, $i=1$}
    \end{subfigure}
    \begin{subfigure}[]{0.3\textwidth}
    \centering
    \includegraphics[scale=0.5]{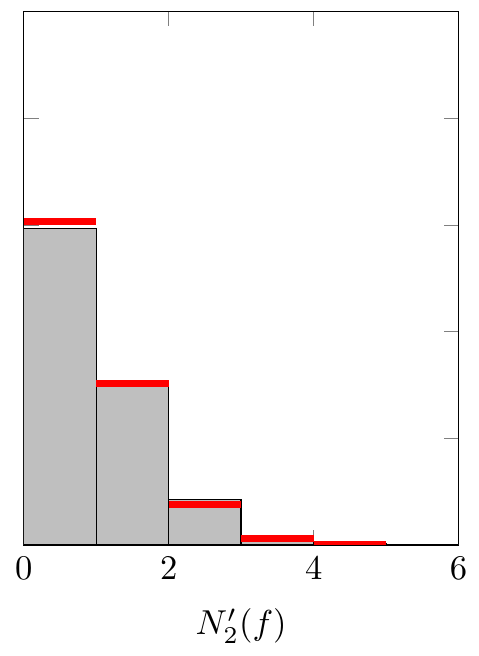}
    \subcaption{$n=50$, $i=2$}
    \end{subfigure}
    \begin{subfigure}[]{0.3\textwidth}
    \centering
    \includegraphics[scale=0.5]{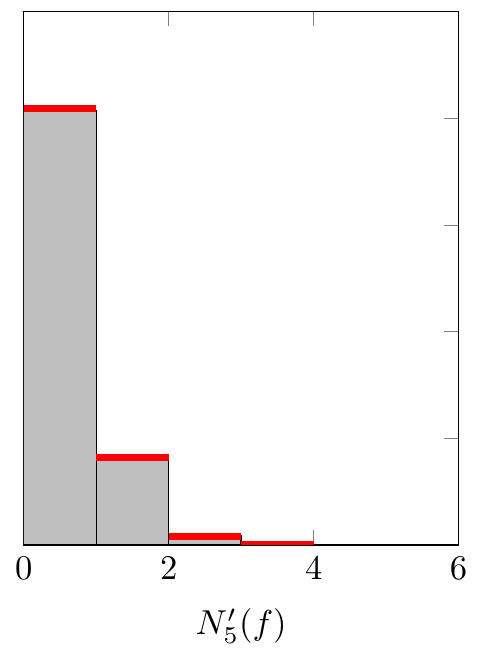}
    \subcaption{$n=50$, $i=5$}
    \end{subfigure}
    \caption{Histograms of $N'_i(f)$ with 10,000 trials when $f$ is a degree $20$ and $50$ random polynomial with coefficients drawn from the uniform measure on $\{-1,0,1\}$ on $\F_p$ for $p=10,000,079$. The expected counts for $\Pois(i^{-1})$ random variables, close approximations to the counts for the uniform model, are shown in red.}
    \label{fig: hist}
\end{figure}

Figure \ref{fig: hist} shows histograms for $N'_i(f)$ with 10,000 trials, when $f$ is a random polynomial of degree $20$ or $50$ with coefficients uniform on $\{-1,0,1\}$. Here, $p\approx 10^7$ and we show data for $i=1,2,5$. Again, the error for $i=1$ seems significantly larger than for larger $i$.

\begin{figure}
    \centering
    \begin{subfigure}[]{0.3\textwidth}
    \centering
    \includegraphics[scale=0.4]{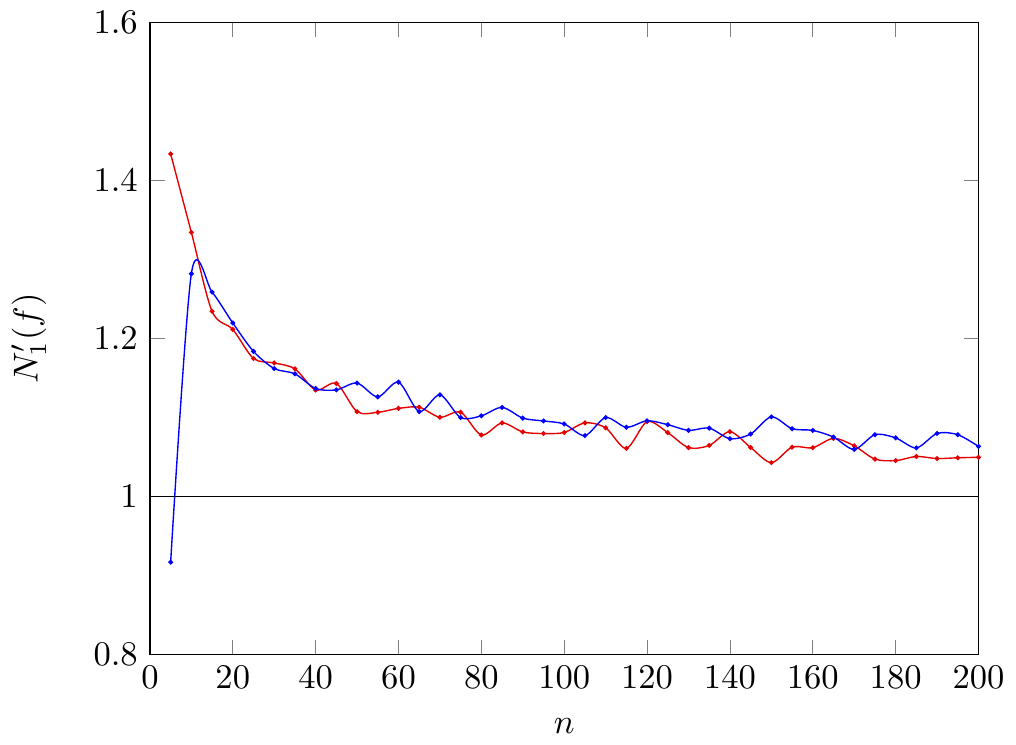}
    \subcaption{$i=1$}
    \end{subfigure}
    \begin{subfigure}[]{0.3\textwidth}
    \centering
    \includegraphics[scale=0.4]{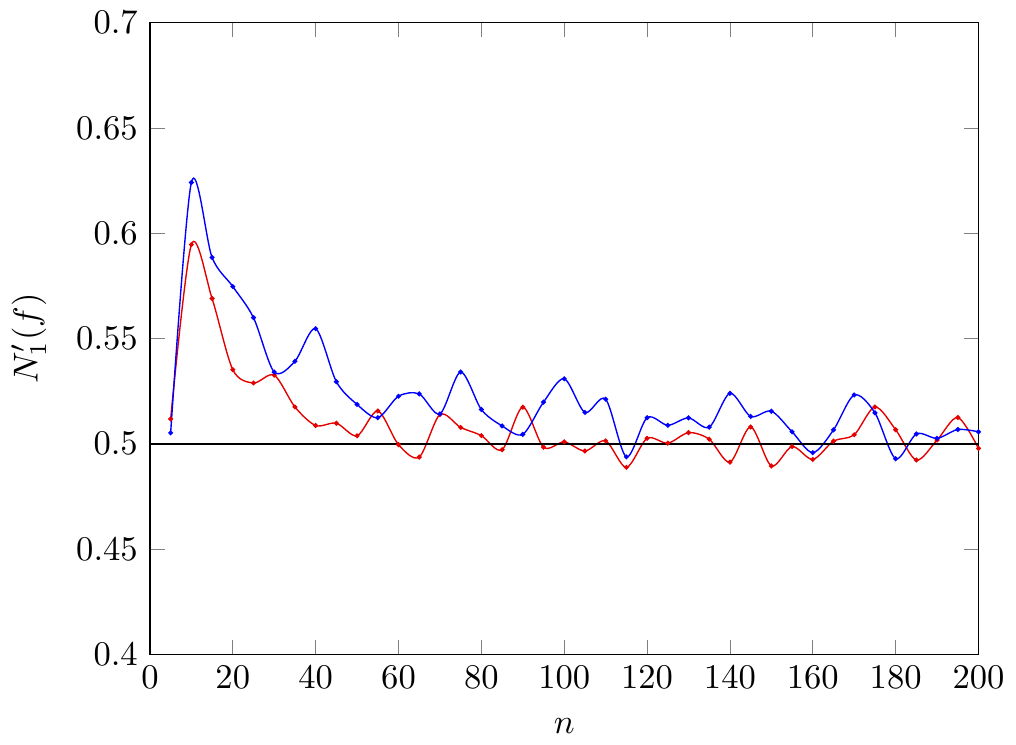}
    \subcaption{$i=2$}
    \end{subfigure}
    \begin{subfigure}[]{0.3\textwidth}
    \centering
    \includegraphics[scale=0.4]{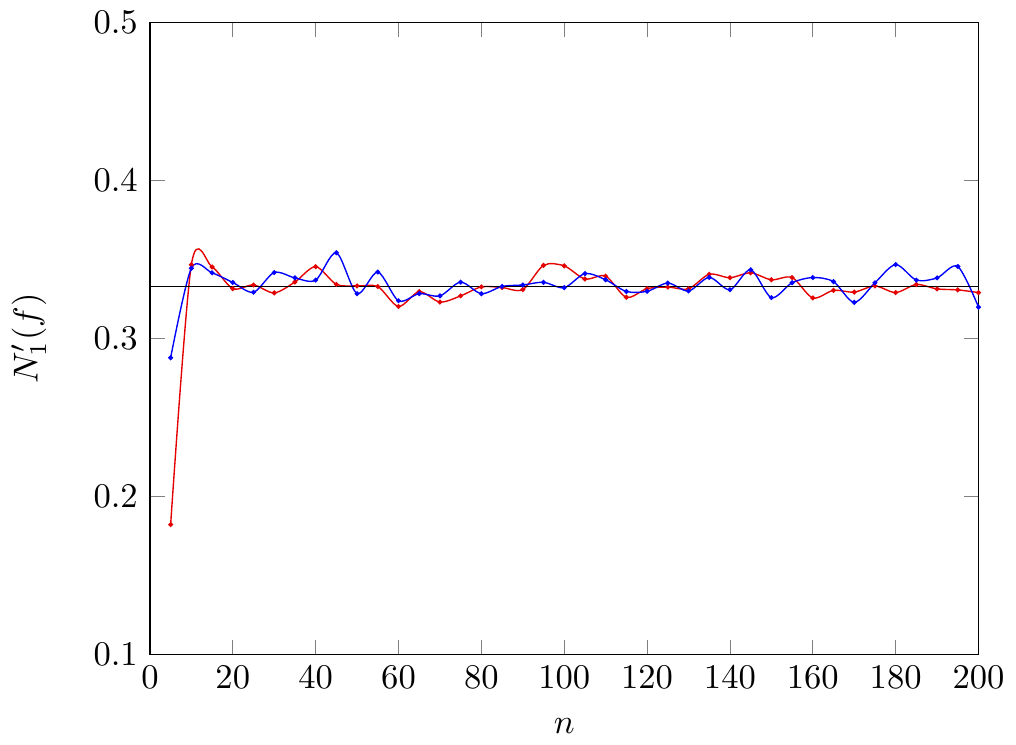}
    \subcaption{$i=5$}
    \end{subfigure}
    \caption{Empirical means of $N'_i(f)$ with 10,000 trials when $\mu$ is the uniform measure on $\{-1,0,1\}$ on $\F_p$ for $f$ a random polynomial of degree $n=5,10,\dotsc,200$. Here $p=101$ is shown in red and $p=10007$ is shown in blue. The black line indicates $\frac{1}{i}$, which is a good approximation for the expected value under the uniform model for large $p$.}
    \label{fig: n dep}
\end{figure}

Figure \ref{fig: n dep} shows the empirical means of $N_i'(f)$ computed with 10,000 trials, for $i=1, 2, 5$, for $p=101$ and $p=10007$. Again, it seems like the error is large for $i=1$ and quickly improves for moderate values of $i$.

This data seems to support the idea that at least for low degree factors, most of the error seems to come from roots of low order. It also seems like the error should note really deteriorate for large $p$, which warrants further study since our methods are unsuited for this.

\subsection{High degree factors}
We now turn to simulations which suggest that even the high degree factors should also exhibit universal behaviour.

\begin{figure}
    \centering
    \begin{subfigure}[]{0.3\textwidth}
    \centering
    \includegraphics[scale=0.7]{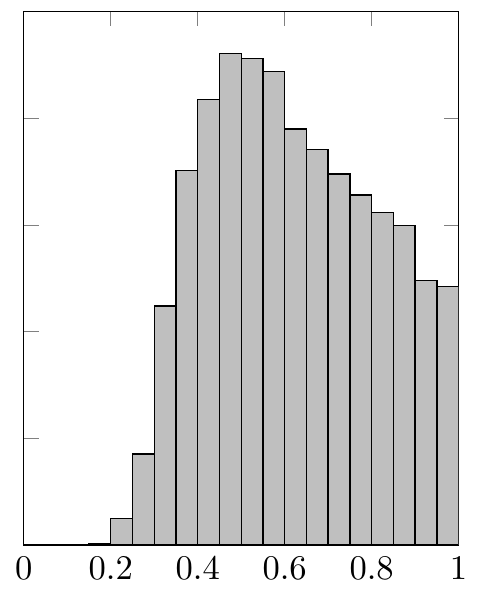}
    \subcaption{$p=11$}
    \end{subfigure}
    \begin{subfigure}[]{0.3\textwidth}
    \centering
    \includegraphics[scale=0.7]{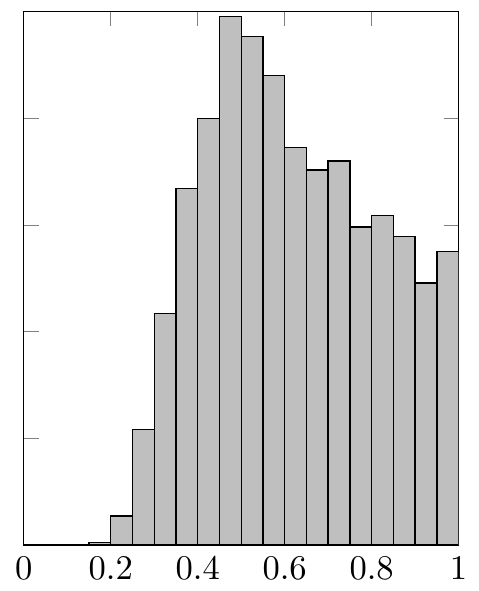}
    \subcaption{$p=10,000,079$}
    \end{subfigure}
    \begin{subfigure}[]{0.3\textwidth}
    \centering
    \includegraphics[scale=0.7]{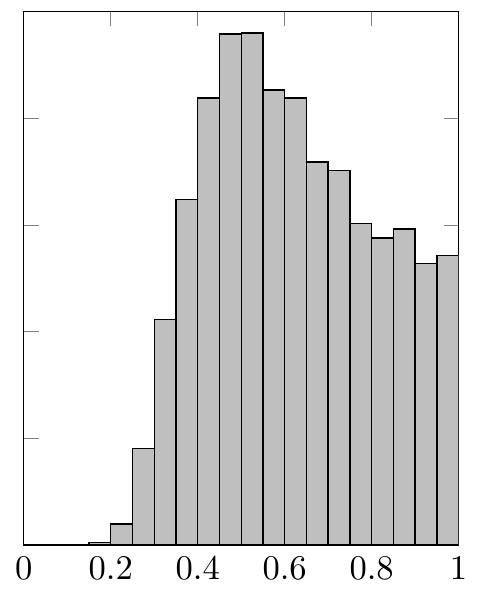}
    \subcaption{Poisson-Dirichlet}
    \end{subfigure}
    \caption{Histograms of the normalized degree of the largest irreducible factor against the maximum of a Poisson-Dirichlet process. Data is shown for 10,000 trials when $\mu$ is the uniform measure on $\{-1,0,1\}$ on $\F_p$ for $f$ a random polynomial of degree $n=500$, and for $p=11$ and $p=10,000,079$.}
    \label{fig: large deg}
\end{figure}

Figure \ref{fig: large deg} shows histograms for the maximal degree of an irreducible factor (normalized by the total degree) against the values for the maximum of a Poisson-Dirichlet process, which is what the uniform model converges to. Based on the data, it seems like at least the maximal degree exhibits universality. While it is harder to check for the joint distribution of all normalized degrees using simulations, it seems plausible that they also exhibit universality.

\section*{Acknowledgments}
The authors thank Nicholas Cook, Persi Diaconis, Sean Eberhard, Ofir Gorodetsky and P\'eter Varj\'u for their help and comments on earlier drafts.

\bibliography{bibliography}{}
\bibliographystyle{habbrv.bst}

\end{document}